\documentclass[12pt]{amsart}
\usepackage{amsmath}
\usepackage{amssymb}
\usepackage{mathtools}
\usepackage{calrsfs}
\usepackage[all]{xy}

\newtheorem{theorem}{Theorem}[section]
\newtheorem{lemma}[theorem]{Lemma}
\newtheorem{proposition}[theorem]{Proposition}
\newtheorem{corollary}[theorem]{Corollary}

\newtheorem*{theorem:taelmanlseriesvarphi}{Theorem~\ref{taelmanlseriesvarphi}}
\newtheorem*{theorem:entire}{Theorem~\ref{entire}}
\newtheorem*{theorem:logalgthm}{Theorem~\ref{logalgthm}}
\newtheorem*{corollary:analog}{Corollary~\ref{analog}}
\newtheorem*{lemma:partialanalog}{Lemma~\ref{partialanalog}}

\theoremstyle{definition}
\newtheorem*{definition}{Definition}
\newtheorem{example}[theorem]{Example}
\newtheorem{remark}[theorem]{Remark}

\theoremstyle{remark}

\newcommand{\ZZ}{\mathbb{Z}}

\newcommand{\RR}{\mathbb{R}}
\newcommand{\NN}{\mathbb{N}}
\newcommand{\TT}{\mathbb{T}}

\newcommand{\CC}{\mathbb{C}}

\DeclarePairedDelimiter\floor{\lfloor}{\rfloor}

\DeclareMathOperator{\Ker}{Ker}
\DeclareMathOperator{\sgn}{sgn}
\DeclareMathOperator{\GL}{GL}

\DeclareMathOperator{\Gal}{Gal}

\DeclareMathOperator{\ord}{ord}

\newcommand{\dnorm}[1]{\lVert #1 \rVert}
\newcommand{\inorm}[1]{{\lvert #1 \rvert}_{\infty}}

\begin{document}

\title[Taelman $L$-values for Drinfeld Modules over Tate Algebras]{Taelman $L$-values for Drinfeld Modules over Tate Algebras}

\author{O\u{g}uz Gezm\.{i}\c{s}}
\address{Department of Mathematics, Texas A{\&}M University, College Station,
TX 77843, U.S.A.}
\email{oguz@math.tamu.edu}
\thanks{This project was partially supported by NSF Grant DMS-1501362}

\subjclass[2010]{Primary 11M38; Secondary 11G09, 11R58}

\date{July 17, 2018}

\begin{abstract}
In the present paper, we investigate Taelman $L$-values corresponding to Drinfeld modules over Tate algebras of arbitrary rank. Using our results, we also introduce an $L$-series converging in Tate algebras which can be seen as a generalization of Pellarin $L$-series. 
\end{abstract}

\keywords{Drinfeld modules, Tate algebras, Pellarin $L$-series, Taelman $L$-values}

\maketitle

\section{Introduction}
Let $p$ be a prime number, and set $q:=p^l$ where $l$ is a positive integer. We denote the finite field with $q$ elements by $\mathbb{F}_q$. Let $\theta$ be an indeterminate over $\mathbb{F}_q$. We set $A:=\mathbb{F}_q[\theta]$ and $K:=\mathbb{F}_q(\theta)$. We define $A_{+}$ to be the set of monic polynomials in $A$, and $A_{+,k}$ to be the subset of $A_{+}$ comprising polynomials of degree $k$. We define $\ord_{\infty}$ to be the valuation on $K$ corresponding to the point at $\infty$ defined so that $\ord_{\infty}(\theta)=-1$. We denote the completion of $K$ with respect to $\ord_{\infty}$ by $\mathbb{K}_{\infty}$, and the completion of the algebraic closure of $\mathbb{K}_{\infty}$ by $\mathbb{C}_{\infty}$. We let $\inorm{\,\cdot\,}$ be the normalized $\infty$-adic norm on $\CC_\infty$ such that $\inorm{\theta} = q$. For any $\mathbb{F}_q$-algebra $R$ with an $\mathbb{F}_q$-algebra homomorphism $\tau \colon R \to R$, we define the non-commutative skew polynomial ring $R\{\tau \}$  and the formal power series ring $R\{\{\tau\}\}$ by the relation $\tau\cdot x=\tau(x) \cdot\tau$ for all $x\in R$. Finally, for any field $k$ and any finite $k[\theta]$-module $M$, we denote the monic generator of the Fitting ideal of $M$ by $[M]_{k[\theta]}$ (see \S 2.2 for the details about Fitting ideals).

A Drinfeld $A$-module $\phi$ of rank $r$ (over $A$) is an $\mathbb{F}_q$-algebra homomorphism 
\[
\phi \colon A \to A\{\tau \}
\]
defined by 
\begin{equation}\label{intro0}
\phi_{\theta}=\theta + \phi_{\theta,1}\tau + \dots + \phi_{\theta,r}\tau^r
\end{equation}
such that $\phi_{\theta,r} \neq 0$. For any integer $s$, the dual Goss $L$-series $L(\phi^{\vee},s-1)$ corresponding to $\phi$ is defined by
\[
L(\phi^{\vee},s-1) =\sum_{a\in A_{+}}\frac{\mu(a)}{a^s},
\]
where $\mu \colon A_{+} \to A$ is a multiplicative function depending on $\phi$ (eg. see \cite[Lem.3.5]{ChangEl-GuindyPapanikolas}). As an example, if $\phi$ is the Carlitz module $C$ defined by $C_{\theta}=\theta + \tau$, then $\mu(a)=1$ for all $a\in A_{+}$. In \cite{Taelman}, Taelman defined a new type of $L$-value $L(\phi,A)$ by
\[
L(\phi,A)=\prod_{f} \frac{\big[A/fA\big]_{A}}{\big[\phi(A/fA)\big]_{A}},
\] 
where the product runs over irreducible polynomials $f\in A_{+}$ (see \S 2.2 and 2.3 for the details). Moreover, Taelman  observed \cite[Rem.~5]{Taelman} that $L(\phi,A)$ is also related to the dual Goss $L$-series by 
\[
L(\phi^{\vee},0)=\sum_{a\in A_{+}}\frac{\mu(a)}{a}=L(\phi,A).
\]

We define the Tate algebra $\TT_{n}$ as the set of all elements of the form $\sum x_{{i_1},\dots,{i_{n}}}z_1^{i_1}\dots z_n^{i_n}\in \CC_{\infty}[[z_1,\dots,z_n]]$ such that $\inorm{x_{{i_1},\dots,{i_{n}}}} \to 0$ as $i_1+\dots+ i_{n}\to \infty$. For $g = \sum x_{{i_1},\dots,{i_{n}}}z_1^{i_1}\dots z_n^{i_n}  \in \TT_{n}$, the Gauss norm $\dnorm{\,\cdot\,}$ on $\TT_{n}$ is defined by $\dnorm{g} := \sup\{\inorm{x_{{i_1},\dots,i_{n}}} \mid i_j\in \NN \}$. We denote its associated valuation by $\ord$ such that $\ord(g) := \min\{ \ord_\infty(x_{{i_1},\dots,{i_{n}}}) \mid i_j \in \NN\}$. For any $f=g/h\in \CC_{\infty}(z_1,\dots,z_n)$ such that $g,h \in \CC_{\infty}[z_1,\dots,z_n]$ with $h\neq 0$, we set $\ord(f):=\ord(g)-\ord(h)$ and denote the completion of $\CC_{\infty}(z_1,\dots,z_n)$ with respect to the valuation $\ord$ by $\widetilde{\TT_{n}}$. Now set $\mathbb{A}:=\mathbb{F}_q(z_1,\dots,z_n)[\theta]$. For $1\leq k \leq n$, define $\ell_0(z_k):=1$, and for $i\geq 1$, $\ell_i(z_k):=\prod_{j=0}^{i-1}(z_k-\theta^{q^j})$. Following the work of Angl\`{e}s and Tavares Ribeiro in \cite[\S 3]{AnglesTavaresRibeiro}, we define a Drinfeld $\mathbb{A}$-module $\varphi$ of rank $r$ as an $\mathbb{F}_q(z_1,\dots,z_n)$-algebra homomorphism 
\[
\varphi \colon \mathbb{A} \to \mathbb{A}\{\tau \}
\]
given by 
\begin{equation}\label{canonical}
\varphi_{\theta}=\sum_{i=0}^{r}\varphi_{\theta,i}\tau^i=\sum_{i=0}^{r}\ell_i(z_1)\dots\ell_i(z_n)\phi_{\theta,i}\tau^i.
\end{equation}
The infinite series $\sum_{i\geq 0} \beta_i\tau^{i}\in \TT_{n}[[\tau]]$ induces the exponential function of $\varphi$
\[
\exp_{\varphi} \colon \widetilde{\TT_{n}} \to \widetilde{\TT_{n}}
\]
defined by $\exp_{\varphi}(f)=\sum_{i\geq 0} \beta_i\tau^{i}(f)$ for all $f\in \widetilde{\TT_{n}}$. Furthermore, the series $\sum_{i\geq 0} \xi_i\tau^{i}\in \TT_{n}[[\tau]]$ induces the logarithm function of $\varphi$
\[
\log_{\varphi} \colon \widetilde{\TT_{n}} \to \widetilde{\TT_{n}}
\]
defined by $\log_{\varphi}(f)=\sum_{i\geq 0} \xi_i\tau^{i}(f)$, and it satisfies $\exp_{\varphi}(\log_{\varphi}(f))=\log_{\varphi}(\exp_{\varphi}(f))=f$ for any $f\in \widetilde{\TT_{n}}$ within the radius of convergence of $\log_{\varphi}$ (see \S 3.1 for more details about coefficients $\beta_i$ and $\xi_i$). 

After Taelman $L$-values were introduced for Drinfeld $A$-modules, Angl\`{e}s, Pellarin, and Tavares Ribeiro developed the theory in \cite[\S 5]{APTR} for the deformation of the Carlitz module $C$. Later on, Angl\`{e}s and Tavares Ribeiro extended the theory for a Drinfeld $\mathbb{A}$-module $\varphi$ of arbitrary rank in \cite{AnglesTavaresRibeiro}, and defined the Taelman $L$-value corresponding to $\varphi$ by the infinite product 
\[
L(\varphi,\mathbb{A})=\prod_{f} \frac{\big[\mathbb{A}/f\mathbb{A}\big]_{\mathbb{A}}}{\big[\varphi(\mathbb{A}/f\mathbb{A})\big]_{\mathbb{A}}},
\] 
where $f$ is irreducible in $A_{+}$ (see \S 4.1 for explicit definitions). Furthermore, Angl\`{e}s, Pellarin, and Tavares Ribeiro \cite[Prop.~5.9]{APTR} proved that if $\tilde{C}$ is a Drinfeld $\mathbb{A}$-module defined by $\tilde{C}_{\theta}=\theta + (z_1-\theta)\dots (z_n-\theta)\tau$, then
\[
L(\tilde{C},\mathbb{A})=\sum_{a\in A_{+}}\frac{a(z_1)\dots a(z_n)}{a}.
\]

In the present paper, we generalize the work of Angl\`{e}s, Pellarin, and Tavares Ribeiro in \cite{APTR} to special Drinfeld $\mathbb{A}$-modules of arbitrary rank which are defined by Angl\`{e}s and Tavares Ribeiro in \cite[\S 3]{AnglesTavaresRibeiro}. Our first result is the following theorem (stated as Theorem \ref{taelmanlseriesvarphi} later). 
\begin{theorem}\label{Theorem1} Let $\phi$ be a Drinfeld $A$-module defined as in \eqref{intro0} and $\varphi$ be a Drinfeld $\mathbb{A}$-module defined as in \eqref{canonical}. Let $L(\varphi,\mathbb{A})$ be the Taelman $L$-value corresponding to $\varphi$. Then
\[ 
L(\varphi,\mathbb{A})=\sum_{a\in A_{+}}\frac{\mu(a)a(z_1)\dots a(z_n)}{a}.
\] 
\end{theorem}

In order to prove Theorem \ref{Theorem1}, we need to analyze $[\varphi(\mathbb{A}/f\mathbb{A})]_{\mathbb{A}}$ for all monic irreducible polynomials $f\in A$. For a Drinfeld $A$-module of arbitrary rank, this was done by using Tate modules (see \cite{ChangEl-GuindyPapanikolas},\cite{Gekeler}). The main difficulty in our case is that we do not yet understand Tate modules for Drinfeld $\mathbb{A}$-modules. Therefore, to prove the theorem, we follow a different direction, and consider Drinfeld $\mathbb{F}_q(z)[\theta]$-modules ($z$ an indeterminate over $\mathbb{F}_q$) which were first introduced by Angl\`{e}s and Tavares Ribeiro in \cite[\S 2.2]{AnglesTavaresRibeiro} (see \S 2.4).

Let $\omega_{n}\in \TT_n^{\times}$ be an Anderson-Thakur type element (see \S 2.1 for the definition). Angl\`{e}s and Tavares Ribeiro \cite[\S 3.2]{AnglesTavaresRibeiro} proved that $\exp_{\varphi}(L(\varphi,\mathbb{A})) \in A[z_1,\dots,z_n]$. Combining this result with Theorem \ref{Theorem1} and by a small calculation (see Remark \ref{calculation}), we deduce the following corollary.
\begin{corollary}\label{Cor1} Let us set
\[
P_{\phi}(z_1,\dots,z_n):=\exp_{\varphi}\bigg(\sum_{a\in A_{+}} \frac{\mu(a)a(z_1)\dots a(z_n)}{a}\bigg).
\]
Then $P_{\phi}(z_1,\dots,z_n)\in A[z_1,\dots,z_n]$. If $P_{\phi}(z_1,\dots,z_n)$ is within the radius of convergence of $\log_{\varphi}$, then
\[
\sum_{a\in A_{+}} \frac{\mu(a)a(z_1)\dots a(z_n)}{a}=\log_{\varphi}(P(z_1,\dots,z_n))=\frac{\log_{\phi}(\omega_{n}P_{\phi}(z_1,\dots,z_n))}{\omega_{n}}.
\]
\end{corollary}
As an immediate consequence of \cite[Rem.~5.13]{APTR}, Angl\`{e}s, Pellarin, and Tavares Ribeiro proved that 
if $0\leq n \leq q-1$, then 
\begin{equation}\label{intro1}
L(\tilde{C},\mathbb{A})=\sum_{a\in A_{+}}\frac{a(z_1)\dots a(z_n)}{a}=\frac{\log_{C}(\omega_n)}{\omega_{n}}.
\end{equation}

In the following result (stated as Corollary \ref{analog} later), we generalize the identity in \eqref{intro1} by analyzing when the polynomial $P_{\phi}(z_1,\dots,z_n)$ in Corollary \ref{Cor1} is equal to 1, and for which Drinfeld $\mathbb{A}$-modules $\varphi$ as in \eqref{canonical}, 1 is within the radius of convergence of $\log_{\varphi}$. Before stating the following result, we set $\beta:=\max\{\deg_{\theta}(\phi_{\theta,i}) \mid 1\leq i \leq r\}$.
\begin{corollary}\label{introcor}
Let $\phi$ be a Drinfeld $A$-module defined as in \eqref{intro0}, and $\varphi$ be a Drinfeld $\mathbb{A}$-module defined as in \eqref{canonical}. If $0\leq n \leq q/r -(1+2\beta)$, then
\begin{equation}
L(\varphi,\mathbb{A})=\sum\limits_{a\in A_{+}}\frac{\mu(a)a(z_1)\dots a(z_n)}{a}=\frac{\log_{\phi}(\omega_n)}{\omega_n}.
\end{equation}
\end{corollary}

The importance of Corollary \ref{introcor} is that it links $\omega_n$, which can be defined in terms of the exponential function of the Carlitz module $C$ (see Example \ref{omegaotherdef}), and the logarithm function of a Drinfeld $A$-module $\phi$ of an arbitrary rank. In other words, the Taelman $L$-value reveals a non-trivial link between Drinfeld $A$-modules over $A$ of arbitrary rank and the Carlitz module.

The reader will no doubt observe that the special value in Theorem \ref{Theorem1} is also the value of a Pellarin $L$-series. In \cite{Pellarin}, Pellarin introduced the following $L$-series
\[
L(z_1,\dots,z_n,s)=\sum_{d\geq 0}\sum_{a\in A_{+,d}}a(z_1)\dots a(z_n)a^{-s}
\]
for any integer $s>0$ which is called \textit{Pellarin $L$-series}. Now, for any $a\in A_{+}$, set $\langle a \rangle := a\theta^{-\deg_{\theta}(a)}$. In \cite{Goss2}, Goss defined another type of Pellarin $L$-series
\begin{equation}\label{Plseries1}
L(z_1,\dots,z_n;x,y)=\sum_{d\geq 0}x^{-d}\sum\limits_{a\in A_{+,d}} a(z_1)\dots a(z_n)\langle a \rangle^y
\end{equation}
for any $(x,y)\in \CC_{\infty}^{\times} \times \ZZ_{p}:=\mathbb{S}_{\infty}$ where we call $\mathbb{S}_{\infty}$ Goss' upper half plane (see \S 4 for details). One can observe that $L( z_1\dots,z_n,s)=L(z_1,\dots,z_n;\theta^s,-s)$. Moreover, when $n=1$, Goss proved \cite[Thm. 1]{Goss2} that $L(z_1,\dots,z_n;x,y)$ is entire on $\CC_{\infty}\times \mathbb{S}_{\infty}$, and Angl\`{e}s and Pellarin \cite[Prop. 6]{AnglesPellarin} proved the same result for arbitrary $n$.

Let $\phi$ be a Drinfeld $A$-module as in \eqref{intro0}. We define 
\begin{equation}\label{Lseries2}
L(\phi^{\vee},z_1,\dots,z_n;x,y)=\sum_{d\geq 0}x^{-d}\sum\limits_{a\in A_{+,d}} \mu(a)a(z_1)\dots a(z_n)\langle a \rangle^y
\end{equation}
for any $(x,y)\in \mathbb{S}_{\infty}$. We further set 
\[
L(\phi^{\vee}, z_1\dots,z_n,s):=L(\phi^{\vee},z_1,\dots,z_n;\theta^s,-s)=\sum_{a\in A_{+}}\frac{\mu(a)a(z_1)\dots a(z_n)}{a^s}.
\]
for any integer $s>0$. Note that $L(C^{\vee},z_1\dots,z_n,s)=L(z_1,\dots,z_n,s)$, and by Theorem \ref{taelmanlseriesvarphi}, we have $L(\phi^{\vee},z_1\dots,z_n,1)=L(\varphi,\mathbb{A})$. 
 
As a generalization of the results mentioned above for Pellarin $L$-series in \eqref{Plseries1}, we prove the following theorem (stated as Theorem \ref{entire} later).
\begin{theorem} \label{introentire} The infinite series $L(\phi^{\vee},z_1,\dots,z_n;x,y)$ can be analytically continued to an entire function on $\mathbb{C}^n_{\infty} \times \mathbb{S}_{\infty}$.
\end{theorem}

One of the reasons to prove Theorem \ref{entire} is that the process of the proof provides us some remarkable results about Drinfeld $A$-modules over $A$ of arbitrary rank. Our method of proving Theorem \ref{entire} relies on a multivariable version of the log-algebraicity theorem for Drinfeld $A$-modules over $A$ of arbitrary rank which can be stated as follows (later stated as Theorem \ref{logalgthm}).
\begin{theorem}\label{intrologalgthm} Let $\phi$ be a Drinfeld $A$-module as in \eqref{intro0}, and let $X_1,\dots, X_n,w$ be indeterminates. The infinite series
\[\exp_{\phi}\bigg(\sum_{a\in A_{+}}\frac{\mu(a)C_a(X_1)\dots C_a(X_n)}{a}w^{q^{\deg_{\theta}(a)}}\bigg) \in K[X_1,\dots,X_n][[w]]
\]
is actually in $A[X_1,\dots,X_n,w]$.
\end{theorem}

The proof of Theorem \ref{intrologalgthm} depends on the method of Angl\`{e}s, Pellarin, and Tavares Ribeiro given in \cite[\S 8]{APTR} to prove the multivariable log-algebraicity theorem for the Carlitz module $C$. We also remark that the one variable version of Theorem \ref{intrologalgthm} was proved by Chang, El-Guindy, and Papanikolas in \cite{ChangEl-GuindyPapanikolas} using the method of Anderson in \cite{And96}.

One of the consequences of Theorem \ref{intrologalgthm} is that we provide vanishing results on power sums twisted by the function $\mu \colon A_{+} \to A$, such as the following one (stated as Lemma \ref{partialanalog} later) used to prove Theorem \ref{introentire}.
\begin{lemma}\label{intropartialanalog} For $ k > r(n + \beta)/(q-1)$, the series
\[
H_{k,n-1}:=\sum_{a\in A_{+,k}} \mu(a)a(z_1)\dots a(z_{n-1}) 
\]
vanishes. 
\end{lemma}

The outline of the paper can be stated as follows. In $\S$ 2, we give further definitions which will be used throughout the paper, and review Taelman $L$-values corresponding to Drinfeld $A$-modules. In $\S$ 3, we calculate Taelman $L$-values corresponding to special Drinfeld modules over Tate algebras introduced by Angl\`{e}s and Tavares Ribeiro in \cite[\S 3]{AnglesTavaresRibeiro}. In $\S$4, we define the $L$-series in \eqref{Lseries2} corresponding to Drinfeld $A$-module $\phi$, and relate its values to Taelman $L$-values. Finally, $\S$ 5 occupies the proof of Theorem \ref{introentire}.
\section*{Acknowledgement} The author is thankful to Matthew A. Papanikolas for useful suggestions and fruitful discussions.

\section{Notations and Preliminaries}
We define the Tate algebra $\TT_{n,t}$ by considering the extra indeterminate $t$, and denote the completion of $\CC_{\infty}(z_1,\dots,z_n,t)$ with respect to the valuation $\ord$ by $\widetilde{\TT_{n,t}}$ .

We define $\tau \colon \widetilde{\TT_{n,t}} \to \widetilde{\TT_{n,t}}$  as a homomorphism of $\mathbb{F}_q(z_1,\dots,z_n,t)$-algebras such that $\tau(\theta)=\theta^q$, and its restriction on $\widetilde{\TT_{n}}$ can be defined similarly. Finally, for any $f\in A$ and any $\alpha\in \CC_{\infty}$, we denote the evaluation of $f$ at $\theta=\alpha$ by $f(\alpha)$, and when $X$ is an indeterminate, we use $f(X)$ for the evaluation of $f$ at $\theta=X$.
\subsection{Anderson-Thakur type elements}
Let $\alpha=(z_1-\theta)(z_2-\theta)\dots (z_n-\theta)\in \TT_{n}$. Choose $y=(-\theta)^{n}$. Then we observe that $\dnorm{y-\alpha}<\dnorm{\alpha}$, and that 
\[
\biggl\lVert \frac{y^{q^j}}{\tau^j(\alpha)} - 1 \biggr\rVert
= \biggl\lVert \frac{\tau^j(y-\alpha)}{\tau^j(\alpha)} \biggr\rVert\to 0
\quad \textup{as} \quad j \to \infty.
\]
Now fix a $(q-1)^{\text{st}}$ root of $y$, and let $\xi=y^{\frac{1}{q-1}}$. Then we define an Anderson-Thakur type element $\omega_n$ by the infinite product
\begin{equation}\label{AndersonThakur}
\omega_n:=\xi \prod_{j=0}^{\infty}\frac{y^{q^j}}{\tau^j(\alpha)}.
\end{equation}
We refer the reader to \cite{APTR} and \cite{GezmisPapanikolas} for further readings about Anderson-Thakur type elements.

\subsection{Fitting Ideals} Let $k$ be any field, and set $R:=k[\theta]$. Let $M$ be a finite $R$-module given as the direct sum $\bigoplus_{i=1}^{n} R/f_iR$ of the quotient modules where $f_i$ is monic in $R$ for $1\leq i \leq n$. Then we set a monic polynomial $\big[M\big]_{R}:= f_1\dot f_2 \dots f_n$, and we call the principal ideal generated by $\big[M\big]_{R}$ in $R$ the Fitting ideal of $M$. As an example, we observe that $\big[R/fR\big]_{R}=f$ for any monic polynomial $f\in R$. We also note that by \cite[\S 2.1]{AnglesTavaresRibeiro}, 
$[M]_{R}$ can be also defined by 
\[\big[M\big]_{R}=\det_{k[X]}((1\otimes X)Id-(\theta\otimes 1) \mid M\otimes_{k} k[X])_{|X=\theta},\]
which can be seen as the evaluation at $X=\theta$ of the characteristic polynomial of the action of $\theta$ on the $R$-module $M\otimes_{k}k$ given by $\theta \cdot (x\otimes 1) =\theta\cdot x\otimes 1$ for all $x\in M$. We refer the reader to \cite{AnglesTavaresRibeiro} and \cite{Taelman} for details about Fitting ideals.

\subsection{Drinfeld $A$-Modules}
Let $\phi$ be a Drinfeld $A$-module defined as in \eqref{intro0}. The exponential series
\[
\exp_{\phi}=\sum_{j\geq 0}\alpha_j\tau^j \in \CC_{\infty}[[\tau]]
\]
is defined so that $\alpha_0=1$ and $\exp_{\phi}a=\phi_{a}\exp_{\phi}$ for all $a\in A$ . The series induces an $\mathbb{F}_q$-linear endomorphism and also an entire function $\exp_{\phi} \colon \CC_{\infty} \to \CC_{\infty}$ defined by $\exp_{\phi}(x)=\sum_{j\geq 0}\alpha_jx^{q^j}$ for all $x\in \CC_{\infty}$.

On the other hand, the logarithm series corresponding to $\phi$ is defined by
\[
\log_{\phi}=\sum_{j\geq 0}\gamma_j\tau^j \in \CC_{\infty}[[\tau]]
\]
such that $\gamma_{0}=1$ and $\log_{\phi}\phi_{a}=a\log_{\phi}$ for all $a\in A$. It induces the logarithm function $\log_{\phi} \colon \CC_{\infty} \to \CC_{\infty}$ which is defined by $\log_{\phi}(x)=\sum_{j\geq 0}\gamma_jx^{q^j}$ for all $x\in \CC_{\infty}$ within the radius of convergence of $\log_{\phi}$. Furthermore, for any $x\in \CC_{\infty}$ where $\log_{\phi}(x)$ is defined, we have $\exp_{\phi}(\log_{\phi}(x))=\log_{\phi}(\exp_{\phi}(x))=x$.
\begin{example}\label{Carlitz} The Carlitz module $C$ defined by $C_{\theta}=\theta + \tau$ is an example of a rank 1 Drinfeld $A$-module. It has the exponential function $\exp_{C}$ so that $\Ker(\exp_{C})=\tilde{\pi}A$ where $\tilde{\pi} \in \CC_{\infty}^{\times}$.
\end{example}
Let $f\in A_{+}$ be an irreducible polynomial. For any element $a\in A$, define $\bar{a}$ such that $a \equiv \bar{a} \pmod f$. We define $\overline{\phi}\colon A\to A/fA[\tau]$ by
\[
\overline{\phi}_{\theta}=\sum_{i=0}^{r_0}\overline{\phi_{\theta,i}}\tau^i
\]
such that $\overline{\phi_{\theta,r_0}}\neq 0$. Observe that $r_0$ depends on the Drinfeld $A$-module $\phi$ and the polynomial $f$. Note also that $0\leq r_0\leq r$. We also set $\phi(A/fA)$ to $A/fA$ whose $A$-module action given by $\theta \cdot f=\overline{\phi}_{\theta}(f)$ for all $f\in A/fA$.

The Taelman $L$-value $L(\phi,A)$ corresponding to $\phi$ is defined by the following Euler product 
\[L(\phi,A)=\prod_{f} \frac{\big[A/fA\big]_{A}}{\big[\phi(A/fA)\big]_{A}},\] 
where the product is over irreducible polynomials of $A_{+}$. Taelman proved \cite[\S 5]{Taelman} that $L(\phi)$ is convergent in $\mathbb{K}_{\infty}$. 

We observe that $\big[A/fA\big]_{A}=f.$ Assume further that $r_0 \geq 1$. Let $h$ be an irreducible polynomial in $A_{+}$ which is not equal to $f$, and let
\[
P_{\phi}(x):=x^{r_0}+p_{r_0-1}x^{r_0-1}+\dots+p_1x+ p_0 \in A[x]
\] 
be the characteristic polynomial of the action of $\tau^d$ on the Tate module $T_{h}(\overline{\phi})$ (see \cite[\S 1]{Gekeler} and \cite[\S 4.10]{Goss} for details about the Tate module). By \cite[Cor.~3.2(a)]{ChangEl-GuindyPapanikolas}, we know that $\deg_{\theta}p_i < d$ for $1\leq i \leq r_0-1$, and $p_0=c(f)^{-1}f$ for some $c(f) \in \mathbb{F}_q^{\times}$. By the work of Gekeler \cite[Thm. 5.1(i)]{Gekeler} (see also \cite[Cor. 3.2(b)]{ChangEl-GuindyPapanikolas}), we know further that $\big[\phi(A/fA)\big]_{A}=c(f)P_{\phi}(1)$.

Define the polynomial $D_{f}^{\phi}(x):=1+c(f)p_1x+c(f)p_2fx^2+\dots + c(f)f^{r_0-1}x^{r_0}$. Following \cite[\S 3]{ChangEl-GuindyPapanikolas}, for any $s\in \ZZ$, the dual Goss $L$-series $L(\phi^{\vee},s-1)$ is given by 
\[L(\phi^{\vee},s-1)=\prod_{f \text{ prime in }A_{+}} D_f^{\phi}(f^{-s})^{-1}=\sum_{a\in A_{+}}\frac{\mu(a)}{a^s},\]
where the function $\mu:A_{+} \rightarrow A$ is defined by the Euler product expansion above and has the generating series (see also \cite[\S 3]{ChangEl-GuindyPapanikolas})
\begin{equation}\label{generating}
\sum_{i=0}^{\infty}\mu(f^i)x^i=  D_f^{\phi}(x)^{-1}.
\end{equation}
\begin{lemma}[{Chang, El-Guindy, Papanikolas \cite[Lem.~3.5]{ChangEl-GuindyPapanikolas}}]\label{mu}Let $a$ be an element in $A_{+}$.
\begin{itemize}
\item [(a)] The function $\mu\colon A_{+} \to A$ is multiplicative.
\item [(b)] $\deg_{\theta}\mu(a)\leq \big(1-\frac{1}{r}\big)\deg_{\theta}(a).$
\end{itemize}
\end{lemma}
\subsection{Drinfeld $\tilde{A}$-Modules}
Throughout this section, we set $z:=z_1$, and abbreviate $\TT_1$ by $\TT$. We also fix the notation $\tilde{A}$ for $\mathbb{F}_q(z)[\theta]$. Let $m\geq 1$ be an integer. Inspired by \cite[\S 2.4]{AnglesTavaresRibeiro}, we define the \textit{$z^m$-deformation of Drinfeld $A$-module $\phi$} as an $\mathbb{F}_q(z)$-algebra homomorphism $\tilde{\phi}\colon \tilde{A} \to \tilde{A}\{\tau\}$ given by 
\begin{equation}\label{definitonzdeform}
 \tilde{\phi}_{\theta}=\sum_{i=0}^r\tilde{\phi}_{\theta,i}\tau^i=\sum_{i=0}^rz^{mi}\phi_{\theta,i}\tau^i .
\end{equation} 
Note that if $r_0=0$, we have that $\tilde{\phi}_{f,j}\equiv 0 \pmod{f}$ for all $j\geq 0$.
\begin{lemma}\label{coefficients1} Let $f$ be an irreducible polynomial in $A_{+}$ of degree $d$ with $r_0\geq 1$, and let $\tilde{\phi}_{f,k}$ be the $k^{th}$ coefficient of $\tilde{\phi}_{f}$. Then for all $0\leq j \leq d-1$, we have
\[
\tilde{\phi}_{f,j}\equiv 0 \pmod{f},
\]
and for $d \leq j \leq r_0d$, we have
\[
\tilde{\phi}_{f,j}\equiv \sum_{i=1}^{\floor*{j/d}}z^{mid}\tilde{\phi}_{-c(f)p_i,j-id} \pmod{f},
\]
where $\floor*{\,\cdot \,}$ is the floor function.
\end{lemma}
\begin{proof}
Observe that for any $k \geq 0$, we have $\tilde{\phi}_{f,k}=z^{mk}\phi_{f,k}$. By \cite[Lem.~5.4(a)]{ChangEl-GuindyPapanikolas}, for $0\leq j \leq d-1$, we get
\begin{equation}\label{congruence}
\phi_{f,j}\equiv 0 \pmod{f}.
\end{equation}
Thus, the first part follows from \eqref{congruence}. By \cite[Lem.~5.4(b)]{ChangEl-GuindyPapanikolas}, we see that for $d \leq j \leq r_0d$,
\[\phi_{f,j}\equiv \sum_{i=1}^{\floor*{j/d}}\phi_{-c(f)p_i,j-id} \pmod{f}.\]
Therefore,
\begin{equation}\label{congruence2}
\tilde{\phi}_{f,j}=z^{mj}\phi_{f,j}\equiv \sum_{i=1}^{\floor*{j/d}}z^{mid}z^{mj-mid}\phi_{-c(f)p_i,j-id}\equiv \sum_{i=1}^{\floor*{j/d}}z^{mid}\tilde{\phi}_{-c(f)p_i,j-id} \pmod{f}.
\end{equation}
\end{proof}
We define $\tilde{\phi}(\tilde{A}/f\tilde{A})$ as an $\tilde{A}$-module $\tilde{A}/f\tilde{A}$ with the induced action of $\tilde{\phi}$ on $\tilde{A}/f\tilde{A}$.
Note that 
\begin{equation}\label{generator}
\big[\tilde{\phi}(\tilde{A}/ f \tilde{A})\big]_{\tilde{A}}=\det_{\mathbb{F}_q(z)[X]}(X-\tilde{\phi}_{\theta} \mid \tilde{A}/f\tilde{A}\otimes_{\mathbb{F}_q(z)} \mathbb{F}_q(z)[X])_{|X=\theta}.
\end{equation}
In other words, $\big[\tilde{\phi}(\tilde{A}/ f \tilde{A})\big]_{\tilde{A}}$ is the evaluation at $X=\theta$ of the characteristic polynomial of the action of $\theta$ on the $\tilde{A}$-module $\tilde{A}/f\tilde{A}\otimes_{\mathbb{F}_q(z)} \mathbb{F}_q(z)$. The monic generator $\big[\tilde{\phi}(\tilde{A}/ f \tilde{A})\big]_{\tilde{A}}$ of the Fitting ideal of $\tilde{\phi}(\tilde{A}/f\tilde{A})$ is an element in $\tilde{A}$. By definition, it is a polynomial in $\theta$ of degree $d$, and is a polynomial in $z^m$ of degree at most $r_0d$. One can observe that if $r_0=0$, then 
\[\big[\tilde{\phi}(\tilde{A}/ f \tilde{A})\big]_{\tilde{A}}=\big[\tilde{A}/ f \tilde{A}]_{\tilde{A}}=f.\]
\begin{proposition}\label{fittingidealzdeformation} For any irreducible polynomial $f\in A_{+}$ of degree $d$ with $r_0\geq 1$, we have
\[
 [\tilde{\phi}(\tilde{A}/ f \tilde{A})]_{\tilde{A}}=f + c(f)p_1z^{md} + c(f)p_2z^{2md} + \dots +c(f)z^{r_0md}.
\]
\end{proposition}
\begin{proof}
To ease the notation, let us define $Q(z):=\big[\tilde{\phi}(\tilde{A}/ f \tilde{A})\big]_{\tilde{A}}$. Note that if we set $z=0$ in \eqref{definitonzdeform}, then $\tilde{\phi}$ becomes the trivial Drinfeld $A$-module meaning that the $A$-module action on $\tilde{\phi}(\tilde{A}/ f \tilde{A})$ is given by multiplying elements in $A/fA$ by $\theta$. Therefore, by \eqref{generator}, we have that $Q(0)=f$. Now let $A_i \in A$ for $1\leq i \leq r_0d$, and set $Q(z)=f + A_1z^m +A_2z^{2m} + \dots + A_{r_0d}z^{r_0md}$. Since $Q(z)$ is a monic polynomial in $\theta$ of degree $d$, we have that $\deg_{\theta}(A_i) < d$ for all $i$. On the other hand, considering $\tilde{A}$-action on the module $ \tilde{A}/ f \tilde{A}$ shows that all elements in $ \tilde{A}/ f \tilde{A}$ is annihilated by $Q(z)$. That is, for any $x\in \tilde{A}/ f \tilde{A}$, we have by Lemma \ref{coefficients1} that
\begin{equation}
\begin{split}
Q(z)\cdot x&=z^m\tilde{\phi}_{A_1}(x) + \dots + z^{md}\tilde{\phi}_{A_{d}}(x) + \dots +z^{r_0md}\tilde{\phi}_{A_{r_0d}}(x)+\tilde{\phi}_f(x) \\
&\equiv z^m\tilde{\phi}_{A_1}(x) + \dots + z^{md}\tilde{\phi}_{A_{d}}(x) + \dots +z^{r_0md}\tilde{\phi}_{A_{r_0d}}(x)+ \sum_{j=d}^{r_0d}\sum_{i=1}^{\floor*{j/d}}z^{mid}\tilde{\phi}_{-c(f)p_i,j-id}(x)  \\
&\equiv 0 \pmod{f}.
\end{split}
\end{equation}
Since $\deg_{\theta}(A_i) < d$ for all $i$, we have that $z^m\tilde{\phi}_{A_1,0}(x)=z^mA_1x=0$. This implies that $A_1=0$. By induction, we see that $A_i=0$ for all $1\leq i \leq d-1$. Now, the coefficient of $z^{md}$ term in $Q(z)\cdot x$ becomes $(\tilde{\phi}_{-c(f)p_1,0} + \tilde{\phi}_{A_{d,0}})x$. But this term is equal to zero for all $x \in \tilde{A}/ f \tilde{A}$. Thus, we have $\tilde{\phi}_{-c(f)p_1,0}=-\tilde{\phi}_{A_{d,0}}$ by the fact the degrees of $p_1$ and $A_d$ are less than $d$. Therefore, we have $c(f)p_1=A_{d}$. Moreover, the coefficient of $z^{m(d+1)}$ term in $Q(z)\cdot x$ appears as $(\tilde{\phi}_{-c(f)p_1,1} + \tilde{\phi}_{A_{d,1}})x^q + \tilde{\phi}_{A_{d+1,0}}x$. Notice that the first two terms cancel out, and the expression is equal to 0 for all $x$. This implies that $A_{d+1}=0$. Similar calculation as above shows that $A_{d+j}=0$ for all $1\leq j \leq d-1$. If we apply the same logic to other coefficients of $Q(z)$, we see that $Q(z)=f + c(f)p_1z^{md} + c(f)p_2z^{2md} + \dots +c(f)z^{r_0md}$ as desired.
\end{proof}

\begin{remark}\label{basis}(See also proof of \cite[Prop. 6.2]{AnglesTaelman}.) Let $f$ be an irreducible polynomial in $A_{+}$ defined by $f=a_0+a_1\theta + a_2\theta^2 + \dots +\theta^d$ such that $r_0\geq 1$. Let $\overline{\mathbb{F}}_q$ be the algebraic closure of $\mathbb{F}_q$. Since the characteristic polynomial is invariant under the extension of scalars, without loss of generality, we can replace the field $\tilde{A}/f\tilde{A}\otimes_{\mathbb{F}_q(z)} \mathbb{F}_q(z)$ in \eqref{generator} by $ F=\tilde{A}/f\tilde{A}\otimes_{\mathbb{F}_q(z)}\overline{\mathbb{F}}_q(z)$.

Let $ \Gal(\tilde{A}/f\tilde{A}/\mathbb{F}_q(z))$ be the Galois group of the field extension $\tilde{A}/f\tilde{A}$ of $\mathbb{F}_q(z)$. We have the isomorphism
\[
F \cong \prod_{g \in \Gal(\tilde{A}/f\tilde{A}/\mathbb{F}_q(z))} \overline{\mathbb{F}}_q(z)
\] 
via a map sending $a \otimes b \in F$ to $(g(a)b)_{g \in S}$ where $S$ is the set of embeddings of $\tilde{A}/f\tilde{A}$ into $\overline{\mathbb{F}}_q(z)$. Observe that $\{1\otimes 1,\bar{\theta}\otimes 1,\bar{\theta^2}\otimes 1,\dots, \bar{\theta}^{d-1}\otimes 1\}$ is an $\overline{\mathbb{F}}_q(z)$-basis for $F$. Furthermore, we have that 
\[
\theta \cdot \begin{bmatrix}
    1\otimes 1 \\
    \bar{\theta}\otimes 1\\
    \vdots  \\
	 \vdots \\
    \bar{\theta}^{d-1} \otimes 1 \\
\end{bmatrix}=P\begin{bmatrix}
    1\otimes 1 \\
    \bar{\theta}\otimes 1\\
    \vdots  \\
	 \vdots \\
    \bar{\theta}^{d-1} \otimes 1 \\
\end{bmatrix}=\begin{bmatrix}
    0 & 1 &0 &\hdots&0 \\
    0& 0 & 1 & \ddots& \vdots \\
    \vdots   & \vdots & \ddots&\ddots & 0 \\
	 \vdots & \vdots & \vdots & 0&1\\
    -a_0& -a_1 & \hdots & \hdots & -a_{d-1}  \\
\end{bmatrix}\begin{bmatrix}
    1\otimes 1 \\
    \bar{\theta}\otimes 1\\
    \vdots  \\
	 \vdots \\
    \bar{\theta}^{d-1} \otimes 1 \\
\end{bmatrix},
\]
where $P$ is the transpose of a companion matrix defined above. Let $\eta \in \overline{\mathbb{F}}_q$ be a fixed root of $f$. Then, there exists $Q\in \GL_{d}(\overline{\mathbb{F}}_q)$ such that 
\[
QPQ^{-1}=\begin{bmatrix}
    \eta & 0 &0 &0&0 \\
    0& \eta^q & 0 & \ddots& 0 \\
    0   & 0 &\ddots  &\ddots & \vdots \\
	 \vdots & \vdots & \ddots & \ddots&0\\
    0& 0 & 0 & 0 & \eta^{q^{d-1}}  \\
\end{bmatrix}.
\]
Now, define $\{v_1,v_2,\dots,v_d\}$ by
\[
\begin{bmatrix}
    v_1\\
    v_2\\
	 \vdots \\
    v_d  \\
\end{bmatrix}=Q\begin{bmatrix}
    1\otimes 1 \\
    \bar{\theta}\otimes 1\\
	 \vdots \\
    \bar{\theta}^{d-1} \otimes 1 \\
\end{bmatrix}.
\]
Therefore, the set $\{ v_1,\dots,v_d\}$ is an $\overline{\mathbb{F}}_q(z)$-basis for $F$ and $\overline{\mathbb{F}}_q(z)[X]$-basis for $F[X]$. Moreover, considering indices modulo $d$, we have 
$\theta \cdot v_i=m_i v_i$ where $m_{i}\in \overline{\mathbb{F}}_q^{\times}$ such that $f(z)=\prod_{i=1}^{d}(z-m_i)$, and $m_i^{q}=m_{i+1}$. Furthermore, we have that
\begin{equation}\label{tauaction}
\tau \cdot v_i =v_{i-1}.
\end{equation}
Thus, for any basis element $v_j$, we have that
\begin{equation}\label{actionofz}
\begin{split}
&(X-\tilde{\phi}_{\theta})\cdot v_j=\bigg(X-\sum_{k=0}^{r_0}z^{mk}\phi_{\theta,k}\tau^k\bigg)\cdot v_j \\
&=\bigg(X-m_{j}-\sum_{i\geq 1,\, r_0-di\geq 0}z^{mdi}\phi_{\theta,di}(m_j)\bigg)v_j+ \bigg(-\phi_{\theta,1}(m_{j-1})z^m -\sum_{i\geq 1,\, r_0-di\geq 1}z^{mdi+m}\phi_{\theta,di+1}(m_{j-1})\bigg)v_{j-1} 
 \\
&+\dots +\bigg(-\phi_{\theta,d-1}(m_{j-(d-1)})z^{md-m}-\sum_{i\geq 1,\, r_0-di\geq d-1}z^{mdi+md-m}\phi_{\theta,di+d-1}(m_{j-(d-1)})\bigg)v_{j-(d-1)}.
\end{split}
\end{equation}
\end{remark}

\section{Taelman $L$-values}
In this section, we analyze Taelman $L$-values corresponding to special Drinfeld modules over Tate algebras which were introduced by Angl\`{e}s and Tavares Ribeiro in \cite[\S 3]{AnglesTavaresRibeiro}. For more information about Drinfeld modules over Tate algebras, we direct the reader to \cite{APTR} and \cite{GezmisPapanikolas}.
\subsection{Drinfeld $\mathbb{A}$-Modules}
In this section, we investigate the Taelman $L$-value corresponding to the Drinfeld $\mathbb{A}$-module $\varphi$ defined in \eqref{canonical}.

We note that the Drinfeld $\mathbb{A}$-module $\varphi$ has an exponential function $\exp_{\varphi}\colon \widetilde{\TT_{n}} \to \widetilde{\TT_{n}}$ which is defined by
\[
\exp_{\varphi}(g)=\sum_{j} \ell_j(z_1)\dots\ell_j(z_n)\alpha_j\tau^j(g),
\]
for all $g\in \widetilde{\TT_{n}}$ where $\alpha_j$'s are coefficients of $\exp_{\phi}$. Note that by \cite[Prop. 3.2.3]{GezmisPapanikolas}, $\exp_{\varphi}$ converges everywhere on $\TT_n$. Its logarithm function $\log_{\varphi}\colon \widetilde{\TT_{n}} \to \widetilde{\TT_{n}}$ is given by
\[
\log_{\varphi}(g)=\sum_{j} \ell_j(z_1)\dots\ell_j(z_n)\gamma_j\tau^j(g)
\]
for all $g\in \widetilde{\TT_{n}}$ where $\gamma_j$'s are coefficients of $\log_{\phi}$ and $g$ within the radius of convergence of $\log_{\varphi}$. Let $f$ be an irreducible element in $A_{+}$, and let $\varphi(\mathbb{A}/f\mathbb{A})$ be the $\mathbb{A}$-module $\mathbb{A}/f\mathbb{A}$ with the induced action of $\varphi$ on $\mathbb{A}/f\mathbb{A}$. Note that 
\begin{equation}\label{trivialaction2}
\big[\mathbb{A}/ f \mathbb{A}\big]_{\mathbb{A}}=f,
\end{equation}
and
\begin{equation}\label{generatorvarphi}
 \big[\varphi(\mathbb{A}/ f \mathbb{A})\big]_{\mathbb{A}}=\det_{\mathbb{F}_q(z_1,\dots,z_n)[X]}(X-\varphi_{\theta} \mid \mathbb{A}/f\mathbb{A}\otimes_{\mathbb{F}_q(z_1,\dots,z_n)} \mathbb{F}_q(z_1,\dots,z_n)[X])_{|X=\theta}.
\end{equation}
The Taelman $L$-value $L(\varphi,\mathbb{A})$ corresponding to $\varphi$ is defined by the following Euler product
\begin{equation}\label{lseriesdef}
L(\varphi,\mathbb{A})=\prod_{f} \frac{\big[\mathbb{A}/f\mathbb{A}\big]_{\mathbb{A}}}{\big[\varphi(\mathbb{A}/f\mathbb{A})\big]_{\mathbb{A}}},
\end{equation}
where the product is over irreducible polynomials of $A_{+}$. By \cite[Thm.~2.7]{Demeslay14} (see also \cite[\S 3.1]{AnglesTavaresRibeiro}), $L(\varphi,\mathbb{A})$ converges in the ring $\mathbb{F}_q(z_1,\dots,z_n)((\theta^{-1}))$.
\begin{remark}\label{basis2}
Let $f$ be a prime in $A_{+}$ of degree $d$ with $r_0\geq 1$. As in Remark \ref{basis}, since the characteristic polynomial is invariant under the extension of scalars, without loss of generality, we can replace the ring 
$\mathbb{A}/f\mathbb{A}\otimes_{\mathbb{F}_q(z_1,\dots,z_n)} \mathbb{F}_q(z_1,\dots,z_n)$ in \eqref{generatorvarphi} by the ring
\begin{equation}\label{thefieldF}
F=\mathbb{A}/f\mathbb{A}\otimes_{\mathbb{F}_q(z_1,\dots,z_n)}\overline{\mathbb{F}}_q(z_1,\dots,z_n)\cong \prod_{g\in \Gal(\mathbb{A}/f\mathbb{A}/\mathbb{F}_q(z_1,\dots,z_n))} \overline{\mathbb{F}}_q(z_1,\dots,z_n).
\end{equation}
Moreover, by the similar calculations in Remark \ref{basis}, we can pick the $\overline{\mathbb{F}}_q(z_1,\dots,z_n)$-basis $\{v_1,\dots,v_n\}$ for $F$ such that \eqref{tauaction} holds. Furthermore, $\theta \cdot v_i=m_i v_i$ where $m_{i}\in \overline{\mathbb{F}}_q^{\times}$ with the property that for $1\leq k \leq n$,
\begin{equation}\label{productwrtbasis}
f(z_k)=\prod_{i=1}^{d}(z_k-m_i)
\end{equation}
and $m_i^{q}=m_{i+1}$ so that the indices are modulo $d$. For $1\leq k \leq n$ and $1\leq i \leq d$, let us define 
\begin{equation}\label{tidef}
t_{k,i}=z_k-m_i.
\end{equation}
Thus, for any basis element $v_j$, we have that
\begin{equation}\label{actionofXandvarphi}
\begin{split}
&(X-\varphi_{\theta})\cdot v_j =\bigg(X-\sum_{i=0}^{r_0}\ell_i(z_1)\dots\ell_i(z_n)\phi_{\theta,i}\tau^i\bigg)\cdot v_j \\
&=\bigg(X-m_{j}-\sum_{i\geq 1,\, r_0-di\geq 0}\phi_{\theta,di}(m_j)\prod_{k=1}^nf(z_k)^i\bigg)v_j+ \\ 
&\bigg(-\phi_{\theta,1}(m_{j-1})\prod_{k=1}^nt_{k,j-1}-\sum_{i\geq 1,\, r_0-di\geq 1}\phi_{\theta,di+1}(m_{j-1})\prod_{k=1}^nt_{k,j-1}f(z_k)^i\bigg)v_{j-1}
+\dots \\
&+\bigg(-\phi_{\theta,d-1}(m_{j-(d-1)})\prod_{k=1}^nt_{k,j-(d-1)}\dots t_{k,j-1} \\
&-\sum_{i\geq 1,\,r_0-di\geq d-1}\phi_{\theta,di+d-1}(m_{j-(d-1)})\prod_{k=1}^nt_{k,j-(d-1)}\dots t_{k,j-1}f(z_k)^i\bigg)v_{j-(d-1)}.
\end{split}
\end{equation}
\end{remark}
We now recall the definition of $r_0$ from \S 2.3 and observe that if $r_0=0$, then
\[\big[\varphi(\mathbb{A}/ f \mathbb{A})\big]_{\mathbb{A}}=\big[\mathbb{A}/ f \mathbb{A}\big]_{\mathbb{A}}=f.\] 
\begin{proposition}\label{fittingidealofcanonical}Let $\varphi$ be a Drinfeld $\mathbb{A}$-module of rank $r$ defined in \eqref{canonical}. Then for all irreducible $f\in A_{+}$ of degree $d$ with $r_0\geq 1$, we have
\[
\big[\varphi(\mathbb{A}/ f \mathbb{A})\big]_{\mathbb{A}}=f + c(f)p_1\prod_{k=1}^nf(z_k)+ c(f)p_2\prod_{k=1}^nf(z_k)^2 + \dots +c(f)\prod_{k=1}^nf(z_k)^{r_0}.
\]
\end{proposition}
\begin{proof}
We extend the idea of Angl\`{e}s and Taelman in the proof of \cite[Prop.~6.2]{AnglesTaelman}. Assume that $n=1$, and fix an irreducible polynomial $f\in A_{+}$ of degree $d$ with $r_0\geq 1$. Set $t_{i}:=t_{1,i}$. We see from \eqref{productwrtbasis} and \eqref{tidef} that
\begin{equation}\label{tiandf}
f(z_1)=\prod_{i=1}^dt_{i}.
\end{equation}
Let $F=\mathbb{A}/f\mathbb{A}\otimes_{\mathbb{F}_q(z_1)}\overline{\mathbb{F}}_q(z_1)$ and $R_{\varphi}$ be the  matrix  representing the map $X-\varphi_{\theta}$ on $F[X]$ with respect to the $\overline{\mathbb{F}}_q(z_1)[X]$-basis $\{v_1,\dots,v_d\}$. We claim that $\det(R_{\varphi})$ is a polynomial in $f(z_1)$ which has constant coefficient equal to $f(X)$. By \eqref{actionofXandvarphi}, it is enough to look at $\det(R_{\varphi})$ when $r_0 < d$. Let $R$ be the matrix defined by
\[
R=[a_{i,j}]:=\begin{bmatrix}
    X-m_{1} & -\phi_{\theta,1}(m_1) &\hdots & -\phi_{\theta,d-1}(m_1) \\
    -\phi_{\theta,d-1}(m_2)& X-m_{2} & \hdots & \vdots \\
    \vdots   & -\phi_{\theta,d-1}(m_3) &\hdots  & \vdots \\
	 \vdots & \vdots & \hdots & -\phi_{\theta,2}(m_{d-2})\\
	 -\phi_{\theta,2}(m_{d-1}) & \vdots & \hdots & -\phi_{\theta,1}(m_{d-1}) \\
    -\phi_{\theta,1}(m_d) & -\phi_{\theta,2}(m_d) & \hdots & X-m_{d}  \\
\end{bmatrix},
\]
and set
\begin{equation}\label{coefficients}
 b_{i,j}:= \left\{
	\begin{array}{ll}
      a_{ij}t_it_{i+1}\dots t_{j-1+d } & \text{if } i>j \\
      a_{ij}\prod_{k=i}^{j-1}t_k & \text{if } i< j \\
      a_{ij} & \text{if } i=j \\
\end{array}. 
\right. 
\end{equation}
By \eqref{actionofXandvarphi}, one can show that $R_{\varphi}=[b_{i,j}]$ when $r_0< d$. Let $S_{d}$ be the symmetric group of degree $d$ and $\sgn(\sigma)$ be the sign of a permutation $\sigma \in S_d$. Then by Leibniz formula, we have
\begin{equation}\label{Leibniz}
\det(R_{\varphi})=\sum_{\sigma \in S_d}\sgn(\sigma)\bigg(\prod_{i=1}^db_{\sigma(i),i}\bigg).
\end{equation}
Notice that  
\begin{equation}\label{productwrtbasisX}
f(X)=\prod_{i=1}^{d}(X-m_i)=\prod_{i=1}^db_{ii}.
\end{equation}
Thus, the identity permutation corresponds to the term $f(X)$ in $\det(R_{\varphi})$. By \eqref{coefficients} and \eqref{Leibniz}, we observe that for any permutation $\sigma \in S_d$, it is enough to look at images of $1\leq i \leq d$ which are not fixed by $\sigma$ to prove the claim. Now let $\sigma=(i_1i_2\dots i_j)$ be a $j$-cycle in $S_d$. We have
\begin{equation}\label{finalpartone}
\begin{split}
\prod_{i \in \{i_1,\dots,i_j\}}b_{\sigma(i),i}&=t_{\sigma(i_j)}t_{\sigma(i_j)+1}\dots t_{i_j-1}t_{\sigma(i_{j-1})}t_{{\sigma(i_{j-1})}+1}\dots t_{i_{j-1}-1}\dots t_{\sigma(i_1)}\dots t_{i_1-1}\prod_{i \in \{i_1,\dots,i_j\}}a_{\sigma(i),i}\\
&=t_{i_1}t_{i_1+1}\dots t_{i_{j}-1}t_{i_j}t_{i_j+1}\dots t_{i_{j-1}-1}\dots t_{i_2}\dots t_{i_{1}-1}\prod_{i \in \{i_1,\dots,i_j\}}a_{\sigma(i),i}
\end{split}
\end{equation}
so that the indices are modulo $d$. Therefore, we see from \eqref{tiandf} and the last line of \eqref{finalpartone} that there exists a positive integer $k_{\sigma}$ such that 
\begin{equation}\label{finalparttwo}
\prod_{i \in \{i_1,\dots,i_j\}}b_{\sigma(i),i}=f(z_1)^{k_{\sigma}}\prod_{i \in \{i_1,\dots,i_j\}}a_{\sigma(i),i}.
\end{equation}
Thus, the claim follows from \eqref{Leibniz}, \eqref{productwrtbasisX}, and \eqref{finalparttwo} together with the fact that every permutation of $S_d$ is a product of disjoint cycles \cite[\S~1.3]{DummitFoote}. 

Using \eqref{actionofz} and \eqref{actionofXandvarphi} after choosing $m=1$, we see that the coefficient of $f(z_1)^i$ term in $\det(R_{\varphi})$ is the coefficient of $z^{di}$ term in $\big[\tilde{\phi}(\tilde{A}/ f \tilde{A})\big]_{\tilde{A}}$. Therefore, by Proposition \ref{fittingidealzdeformation}, we have that 
\[
\det(R)=f(X) + c(f)p_1(X)f(z_1)+ c(f)p_2(X)f(z_1)^2 + \dots +c(f)f(z_1)^{r_0}.
\]
The proposition follows from evaluating $\det(R)$ at $X=\theta$ when $n=1$.

Now for arbitrary $n$, we replace the field $F$ by $\mathbb{A}/f\mathbb{A}\otimes_{\mathbb{F}_q(z_1,\dots,z_n)}\overline{\mathbb{F}}_q(z_1,\dots,z_n)$. Then, to finish the proof, we apply the same argument above by choosing $m=n$ in \eqref{actionofz}
and noticing that the coefficient of $\prod_{k=1}^nf(z_k)^i$ term in $\det(R_{\varphi})$ is the coefficient of $z^{ndi}$ term in $\big[\tilde{\phi}(\tilde{A}/ f \tilde{A})\big]_{\tilde{A}}$.
\end{proof}
\begin{theorem}\label{taelmanlseriesvarphi} Let $\varphi$ be a Drinfeld $\mathbb{A}$-module as in \eqref{canonical}, and $L(\varphi,\mathbb{A})$ be the Taelman $L$-value corresponding to $\varphi$ defined in \eqref{lseriesdef}. Then
\[ L(\varphi,\mathbb{A})=\prod_{f}\frac{\big[\mathbb{A}/f\mathbb{A}\big]_{\mathbb{A}}}{\big[\varphi(\mathbb{A}/ f \mathbb{A})\big]_{\mathbb{A}}}=\sum_{a\in A_{+}}\frac{\mu(a)a(z_1)\dots a(z_n)}{a},
\] where the product runs over irreducible polynomials $f$ of $A_{+}$.
\end{theorem}
\begin{proof} For any irreducible polynomial $f\in A_{+}$ with $r_0\geq 1$, define the polynomial 
\begin{align*}
D^{\varphi}_f(x)&:=D_f^{\phi}(f(z_1)\dots f(z_n)x) \\
&=1+c(f)p_1\prod_{i=1}^nf(z_i)x+ c(f)p_2f\prod_{i=1}^nf(z_i)^2x^2 + \dots + c(f)f^{r_0-1}\prod_{i=1}^nf(z_i)^{r_0}x^{r_0}.
\end{align*}
For $s\geq 1$, define the following $L$-series by 
\begin{equation}\label{lseriesdef1}
L(\varphi^{\vee},s-1):=\prod_{f} D^{\varphi}_f(f^{-s})^{-1} =\sum_{a\in A_{+}}\frac{\mu_{\varphi}(a)}{a^s},
\end{equation}
where the product is over irreducible polynomials $f$ in $A_{+}$, and $\mu_{\varphi} \colon A_{+} \to \mathbb{A}$ is a function satisfying \eqref{lseriesdef1}. Using \eqref{trivialaction2} and Proposition \ref{fittingidealofcanonical}, we see that 
\begin{align*}
\frac{\big[\mathbb{A}/f\mathbb{A}\big]_{\mathbb{A}}}{\big[\varphi(\mathbb{A}/f\mathbb{A})\big]_{\mathbb{A}}}&=\frac{f}{
f + c(f)p_1\prod_{i=1}^nf(z_i)+ c(f)p_2\prod_{i=1}^nf(z_i)^2 + \dots +c(f)\prod_{i=1}^nf(z_i)^{r_0}}\\
&=\frac{1}{1 + c(f)p_1f^{-1}\prod_{i=1}^nf(z_i)+ \dots +c(f)f^{r_0-1}f^{-r_0}\prod_{i=1}^nf(z_i)^{r_0}}\\
&=D^{\varphi}_f(f^{-1})^{-1}.
\end{align*}
Thus, by \eqref{lseriesdef} and \eqref{lseriesdef1}, $L(\varphi,\mathbb{A})=L(\varphi^{\vee},0)$.
Observe that \eqref{generating} implies 
\begin{equation}\label{genseriesdef}
\sum_{i=0}^{\infty}(f(z_1)\dots f(z_n))^i\mu(f^i)x^i=D_f^{\phi}(f(z_1)\dots f(z_n)x)^{-1}=D_f^{\varphi}(x)^{-1}.
\end{equation}
Since by Lemma \ref{mu}(a), $\mu \colon A_{+} \to A$ is a multiplicative function, \eqref{genseriesdef} implies that $\mu_{\varphi}(a)=\mu(a)a(z_1)\dots a(z_n)$ for all $a\in A_{+}$. Thus,
\[ 
L(\varphi,\mathbb{A})=L(\varphi^{\vee},0)=\prod_{f \text{ irreducible in }A_{+}} D^{\varphi}_f(f^{-1})^{-1}=\sum_{a\in A_{+}}\frac{\mu(a)a(z_1)\dots a(z_n)}{a}.
\]
\end{proof}

\subsection{Drinfeld $\tilde{\mathbb{A}}$-Modules}
Let $\phi$ be a Drinfeld $A$-module as in \eqref{intro0} and $t$ be an indeterminate over $\CC_{\infty}$. Let $\tilde{\mathbb{A}}$ be the polynomial ring $\mathbb{F}_q(z_1,\dots,z_n,t)[\theta]$. We define the Drinfeld $\tilde{\mathbb{A}}$-module $\psi$ as an $\mathbb{F}_q(z_1,\dots,z_n,t)$-algebra homomorphism $\psi \colon \tilde{\mathbb{A}} \to \tilde{\mathbb{A}}\{\tau\}$ by 
\begin{equation}\label{psi}
\psi_{\theta}=\sum_{i=1}^r\psi_{\theta,i}\tau^i=\sum_{i=1}^rt^i\ell_i(z_1)\dots\ell_i(z_n)\phi_{\theta,i}\tau^i.
\end{equation}
It has the exponential series $\exp_{\psi}=\sum_{i\geq 0} \alpha_j\ell_j(z_1)\dots\ell_j(z_n)t^j\tau^j \in \TT_{n,t}[[\tau]]$ which induces the exponential function $\exp_{\psi}\colon \widetilde{\TT_{n,t}} \to \widetilde{\TT_{n,t}}$ that converges on $\TT_{n,t}$ by \cite[Prop. 3.2.3]{GezmisPapanikolas}, and is defined by
\[
\exp_{\psi}(g)=\sum_{j} \ell_j(z_1)\dots\ell_j(z_n)t^j\alpha_j\tau^j(g),
\]
for all $g\in\widetilde{\TT_{n,t}}$. Similarly, it has the logarithm series $\log_{\psi}=\sum_{i\geq 0} \gamma_j\ell_j(z_1)\dots\ell_j(z_n)t^j\tau^j \in \TT_{n,t}[[\tau]]$, and it induces the logarithm function  $\log_{\psi} \colon \widetilde{\TT_{n,t}} \to \widetilde{\TT_{n,t}}$ which is defined by
\[
\log_{\psi}(g)=\sum_{j} \ell_j(z_1)\dots\ell_j(z_n)t^j\gamma_j\tau^j(g)
\] 
for all $g\in \widetilde{\TT_{n,t}}$ in the domain of $\log_{\psi}$. We construct the Taelman $L$-value $L(\psi,\tilde{\mathbb{A}})$ by
\begin{equation}\label{taelmanlseriespsi}
L(\psi,\tilde{\mathbb{A}})=\prod_{f}\frac{\big[\tilde{\mathbb{A}}/f\tilde{\mathbb{A}}\big]_{\tilde{\mathbb{A}}}}{\big[\psi(\tilde{\mathbb{A}}/ f \tilde{\mathbb{A}})\big]_{\tilde{\mathbb{A}}}},
\end{equation}
where the product runs over irreducible polynomials $f$ of $A_{+}$. We remark that by \cite[Thm.~2.7]{Demeslay14} (see also \cite[\S 3.1]{AnglesTavaresRibeiro}), $L(\psi,\tilde{\mathbb{A}})$ converges in the ring $\mathbb{F}_q(z_1,\dots,z_n,t)((\theta^{-1}))$. We observe that 
\begin{equation}\label{generatorpsi}
 \big[\psi(\tilde{\mathbb{A}}/ f \tilde{\mathbb{A}})\big]_{\tilde{\mathbb{A}}}=\det_{\mathbb{F}_q(z_1,\dots,z_n,t)[X]}(X-\psi_{\theta} \mid \tilde{\mathbb{A}}/f\tilde{\mathbb{A}}\otimes_{\mathbb{F}_q(z_1,\dots,z_n,t)} \mathbb{F}_q(z_1,\dots,z_n,t)[X])_{|X=\theta}.
\end{equation}
\begin{remark}\label{actionofXandpsi} Let $f$ be a monic irreducible polynomial in $A$ of degree $d$ with $r_0\geq 1$. Following the same idea in Remark \ref{basis} and \ref{basis2}, we can replace the field $\tilde{\mathbb{A}}/f\tilde{\mathbb{A}}\otimes_{\mathbb{F}_q(z_1,\dots,z_n,t)} \mathbb{F}_q(z_1,\dots,z_n,t)$ in \eqref{generatorpsi} with $F:=\tilde{\mathbb{A}}/f\tilde{\mathbb{A}}\otimes_{\mathbb{F}_q(z_1,\dots,z_n,t)} \overline{\mathbb{F}}_q(z_1,\dots,z_n,t)$. Since $t$ is an indeterminate over $\overline{\mathbb{F}}_q(z_1,\dots,z_n)$, the $\overline{\mathbb{F}}_q(z_1,\dots,z_n)$-basis $\{v_1,\dots,v_d\}$ in Remark \ref{basis2} is also an $\overline{\mathbb{F}}_q(z_1,\dots,z_n,t)$-basis for $F$ with the same properties. Thus, we conclude that for $1\leq j \leq d$, 
\begin{equation}
\begin{split}
&(X-\psi_{\theta})\cdot v_j =\bigg(X-\sum_{i=0}^{r_0}\ell_i(z_1)\dots\ell_i(z_n)t^i\phi_{\theta,i}\tau^i\bigg)\cdot v_j \\
&=\bigg(X-m_{j}-\sum_{i\geq 1,\,r_0-di\geq 0}\phi_{\theta,di}(m_j)\prod_{k=1}^nt^{di}f(z_k)^i\bigg)v_j+ \\ 
&\bigg(-\phi_{\theta,1}(m_{j-1})t\prod_{k=1}^nt_{k,j-1}-\sum_{i\geq 1,\, r_0-di\geq 1}\phi_{\theta,di+1}(m_{j-1})\prod_{k=1}^nt_{k,j-1}t^{di+1}f(z_k)^i\bigg)v_{j-1}
+\dots \\
&+\bigg(-\phi_{\theta,d-1}(m_{j-(d-1)})t^{d-1}\prod_{k=1}^nt_{k,j-(d-1)}\dots t_{k,j-1} \\
&-\sum_{i\geq 1,\, r_0-di\geq d-1}\phi_{\theta,di+d-1}(m_{j-(d-1)})\prod_{k=1}^nt_{k,j-(d-1)}\dots t_{k,j-1}t^{di+d-1}f(z_k)^i\bigg)v_{j-(d-1)}.
\end{split}
\end{equation}
\end{remark}

\begin{theorem}\label{canonicaldeformationTaelman}Let $\psi$ be the Drinfeld $\tilde{\mathbb{A}}$-module defined in \eqref{psi}. Then
\[
L(\psi,\tilde{\mathbb{A}})=\sum_{a\in A_{+}}\frac{\mu(a)a(z_1)\dots a(z_n) t^{\deg_{\theta}(a)}}{a}.
\]
\end{theorem}

\begin{proof} We do the proof for $n=1$, and the multivariable version follows similarly. Let $f$ be an irreducible in $A_{+}$ of degree $d$. One can observe that if $r_0=0$, then
\[\big[\psi(\tilde{\mathbb{A}}/ f \tilde{\mathbb{A}})\big]_{\tilde{\mathbb{A}}}=\big[\tilde{\mathbb{A}}/ f \tilde{\mathbb{A}}\big]_{\tilde{\mathbb{A}}}=f.\] 
Assume that $r_0\geq 1$. To calculate $\big[\psi(\tilde{\mathbb{A}}/ f \tilde{\mathbb{A})}\big]_{\tilde{\mathbb{A}}}$, by using Remark \ref{actionofXandpsi} and the similar idea of the proof of Proposition \ref{fittingidealofcanonical}, we see that it is enough to evaluate the determinant of the matrix $Q$ defined by
\[Q:=\begin{bmatrix}
    X-m_1 & -\phi_{\theta,1}(m_1)t_1t &\hdots &-\phi_{\theta,d-1}(m_1)t_1\dots t_{d-1} t^{d-1} \\
    -\phi_{\theta,d-1}(m_2)t_2\dots t_dt^{d-1} & X-m_2 & \hdots & \vdots \\
    \vdots   & \vdots & \hdots & \vdots \\
	 \vdots & \vdots & \hdots & \vdots\\
	 -\phi_{\theta,2}(m_{d-1})t_{d-1}t_{d}t^2 & \vdots & \hdots & -\phi_{\theta,1}(m_{d-1})t_{d-1}t \\
    -\phi_{\theta,1}(m_d)t_dt& -\phi_{\theta,2}(m_d)t_dt_1t^2 & \hdots & X-m_{d}  \\
\end{bmatrix}\]
where $t_{1,i}=t_i$. Using the same argument in the proof of Proposition \ref{fittingidealofcanonical}, we see that 
\begin{align*}
\big[\psi(\tilde{\mathbb{A}}/f\tilde{\mathbb{A}})\big]_{\tilde{\mathbb{A}}}&=\det(Q) \\
&=f + c(f)p_1t^df(z_1)+ c(f)p_2t^{2d}f(z_1)^2 + \dots +c(f)t^{r_0d}f(z_1)^{r_0}.
\end{align*}
Now let us define 
\begin{align*}
D^{\psi}_f(x):&=D^{\phi}_{f}(t^df(z_1)x)\\
&=1+c(f)p_1t^df(z_1) x+ c(f)p_2t^{2d}f(z_1)^2fx^2 + \dots + c(f)t^{dr_0}f(z_1)^{dr_0}f^{r_0-1}x^{r_0},
\end{align*}
and for $s\geq 1$, define 
\begin{equation}\label{lseriesforpsi}
 L(\psi^{\vee},s-1):=\prod\limits_{f} D^{\psi}_f(f^{-s})^{-1}=\sum_{a\in A_{+}}\frac{\mu_{\psi}(a)}{a^s},
\end{equation}
where $f$ runs over primes of $A_{+}$ and $\mu_{\psi} \colon A_{+} \to \tilde{\mathbb{A}}$ is a function satisfying \eqref{lseriesforpsi}. We have
\begin{align*}
\frac{\big[\tilde{\mathbb{A}}/f\tilde{\mathbb{A}}\big]_{\tilde{\mathbb{A}}}}{\big[\psi(\tilde{\mathbb{A}}/f\tilde{\mathbb{A}})\big]_{\tilde{\mathbb{A}}}}&=\frac{f}{
f + c(f)p_1t^df(z_1)+ c(f)p_2t^{2d}f(z_1)^2 + \dots +c(f)t^{dr_0}f(z_1)^{r_0}}\\
&=\frac{1}{1 + c(f)p_1t^df(z_1)f^{-1}+ \dots +c(f)t^{r_0d}f^{r_0-1}f(z_1)^{r_0}f^{-r_0}}\\
&=D^{\psi}_f(f^{-1})^{-1}.
\end{align*}
Thus, by \eqref{taelmanlseriespsi}, $ L(\psi,\tilde{\mathbb{A}})=L(\psi^{\vee},0)$. Note that by \eqref{generating}, we get
\[
\sum_{i=0}^{\infty}(f(z_1)t^d)^i\mu(f^i)x^i=D_f^{\phi}(f(z_1)t^dx)^{-1}=D_f^{\psi}(x)^{-1}.
\]
Therefore, Lemma \ref{mu}(a) implies that $\mu_{\psi}(a)=a(z_1)t^{\deg_{\theta}(a)}\mu(a) $ for all $a\in A_{+}$ and that
\[ 
 L(\psi,\tilde{\mathbb{A}})=L(\psi^{\vee},0)=\sum\limits_{a\in A_{+}}\frac{\mu(a)a(z_1)t^{\deg_{\theta}(a)}}{a}.
\]
\end{proof}
Combining Theorem \ref{canonicaldeformationTaelman} with the result of Angl\`{e}s and Tavares Ribeiro \cite[Prop. 5]{AnglesTavaresRibeiro}, we deduce the following proposition.
\begin{proposition}\label{Prop5} We have 
\[
\exp_{\psi}(L(\psi,\tilde{\mathbb{A}}))=\exp_{\psi}\bigg(\sum_{a\in A_{+}}\frac{\mu(a)a(z_1)\dots a(z_n) t^{\deg_{\theta}(a)}}{a}\bigg) \in A[z_1,\dots,z_n,t].
\]
\end{proposition}

\section{$L$-Series $L(\phi^{\vee},z_1,\dots,z_n;x,y)$}
We consider the topological group $\mathbb{S}_{\infty}:=\mathbb{C}_{\infty}^{\times} \times \mathbb{Z}_p$ where the group action is given by addition componentwise. For any element $a\in A$, we set 
\[
\langle a \rangle:=a\theta^{-\deg_{\theta}(a)}\in 1 + \frac{1}{\theta} \mathbb{F}_q\bigg[\frac{1}{\theta}\bigg].
\]
For any $y\in \mathbb{Z}_p$, we define the exponentiation of $\langle a \rangle$ by
\[\langle a \rangle^{y}:=\sum\limits_{i\geq 0} \binom{y}{i} (\langle a \rangle -1)^{i}, \]
where the binomial $\binom{y}{i}$ is defined by the Lucas' formula (See \cite[Thm.~1]{Fine}). In particular, as one can also see in \cite[\S~7]{Brownawell}, if the $p$-adic expansion of $y$ is $\sum\beta_jp^j$ where $\beta_j \in \{0,1,\dots,p-1\}$, and the $p$-adic expansion of $i$ is $\sum_{k=0}^n i_kp^k$ where $i_k\in \{0,1,\dots,p-1\}$, then we define 
\begin{equation}\label{Lucasformula}
\binom{y}{i}:=\prod_{j=0}^n  \binom{\beta_j}{i_j}. 
\end{equation}
We also remark that since $\langle a \rangle$ is a 1-unit for any $a\in A$, $\langle a \rangle^{y}$ converges in $\mathbb{K}_{\infty}$ for any $y\in \mathbb{Z}_p$. 

Let $\phi$ be a Drinfeld $A$-module of rank $r$ defined as in \eqref{intro0}. For any $(x,y)\in \mathbb{S}_{\infty}$, consider the series 
\[
L(\phi^{\vee},z_1,\dots,z_n;x,y):=\sum_{d \geq 0} \mathcal{L}_{d,n}(x,y)(z_1,\dots,z_n),
\]
where we set
\[
\mathcal{L}_{d,n}(x,y)(z_1,\dots,z_n):=x^{-d}\sum\limits_{a\in A_{+,d}} \mu(a)a(z_1)\dots a(z_n)\langle a \rangle^y.
\]
For any integer $s>0$, we construct the following $L$-Series by
\[
L(\phi^{\vee},z_1,\dots,z_n,s):=L(\phi^{\vee},z_1,\dots,z_n;\theta^{s},-s)=\sum\limits_{a \in A_{+}}\frac{\mu(a)a(z_1)\dots a(z_n)}{a^s}.
\]
Note that by Lemma \ref{mu}(b), $L(\phi^{\vee},z_1,\dots,z_n,s)$ converges in $\TT_{n}$ for any integer $s \geq 1$ . 
\subsection{The Value of $L(\phi^{\vee},z_1,\dots,z_n,s)$ at $s=1$}
Let $\phi$ be a Drinfeld $A$-module of rank $r$ as in \eqref{intro0}, $\varphi$ be the Drinfeld $\mathbb{A}$-module as in \eqref{canonical} and $\psi$ be the Drinfeld $\tilde{\mathbb{A}}$-module defined as in \eqref{psi}. Recall that $\beta=\max\{\deg_{\theta}(\phi_{\theta,i}) \mid 1\leq i \leq r\}$. By Theorem \ref{taelmanlseriesvarphi}, we have
\[L(\phi^{\vee},z_1,\dots,z_n,1)=L(\varphi,\mathbb{A}).\]
In this section, our aim is to relate the Taelman $L$-value $L(\varphi,\mathbb{A})$ to the logarithm function $\log_{\varphi}$. First, we need a proposition.

\begin{proposition} If $0\leq n \leq q/r -(1+2\beta)$, then 
\begin{equation}\label{integrality}
\exp_{\psi}\bigg(\sum\limits_{a\in A_{+}}\frac{\mu(a)a(z_1)\dots a(z_n)t^{\deg_{\theta}(a)}}{a}\bigg)=1.
\end{equation}
In particular,
\[\exp_{\varphi}\bigg(\sum\limits_{a\in A_{+}}\frac{\mu(a)a(z_1)\dots a(z_n)}{a}\bigg)=1.\]
\end{proposition}
\begin{proof}
We adapt the ideas in \cite[Lem.~7]{AnglesTavaresRibeiro}. Let $\alpha_j$ and $\alpha_j^{\prime}$ be the coefficients of $\exp_{\phi}$ and $\exp_{\psi}$ respectively. By \cite[Eq.~28]{ElGuindyPapanikolas}, we have for all $j\geq 1$ that 
\[
\ord_{\infty}(\alpha_j) \geq q^{j}\bigg(\frac{j}{r}-\frac{\beta}{q-1}\bigg).
\]
This means that 
\begin{equation}\label{orderbound}
\ord(\alpha^{\prime}_j)=\ord(\alpha_j\ell_j(z_1)\dots \ell_j(z_n)t^j) \geq q^{j}\bigg(\frac{j}{r}-\frac{\beta+n}{q-1}\bigg)+\frac{n}{q-1}.
\end{equation}
Let $h\colon \RR_{\geq 1} \to \RR$ be a function defined by
\[
h(x)=q^{x}\bigg(\frac{x}{r}-\frac{\beta+n}{q-1}\bigg) +\frac{n}{q-1}.
\]
Note that $h$ is an increasing function on $\RR_{\geq 1}$ when $\beta+n \leq (q-1)/r$. Moreover, if $\beta+n \leq -1+(n+1)/q+(q-1)/r$ and $q > n$, then for all $x\geq 1$, we have
\begin{equation}\label{inequality}
\begin{split}
h(x)=q^{x}\bigg(\frac{x}{r}-\frac{\beta+n}{q-1}\bigg) + \frac{n}{q-1} &\geq q\bigg(\frac{1}{r}-\frac{\beta+n}{q-1}\bigg)+ \frac{n}{q-1} \\
&\geq q\bigg(\frac{1}{r} + \frac{1}{q-1} -\frac{n+1}{q(q-1)}-\frac{1}{r}\bigg)+ \frac{n}{q-1}\\
&=1.
\end{split}
\end{equation}
Since $q \geq 2$ and $r\geq 1$, if $0\leq n \leq q/r -(1+2\beta)$, we note that 
\begin{equation}\label{ineqcalc}
\begin{split}
n \leq \frac{q}{r} -(1+2\beta) &\Longleftrightarrow q - (rn+r+\beta r) \geq \beta r\\
& \Longrightarrow (q-(rn+r+\beta r))(q-1) - \beta r \geq 0 \\
& \Longleftrightarrow q^2-q(rn+r+\beta r +1) + r(n+1) \geq 0 \\
& \Longleftrightarrow -1 + \frac{n+1}{q}+\frac{q-1}{r} \geq \beta + n,
\end{split}
\end{equation}
and $q > n$ so that \eqref{inequality} follows. Thus, \eqref{orderbound} and \eqref{inequality} imply that $\ord(\alpha_j^{\prime})\geq 1$ for all $j\geq 1$. On the other hand, by Lemma \ref{mu}(b), we know that 
\[
\ord\bigg(\sum\limits_{a\in A_{+}}\frac{\mu(a)a(z_1)\dots a(z_n)t^{\deg_{\theta}(a)}}{a}-1\bigg) \geq 1.
\]
Thus,
\begin{equation}\label{integrality2}
\ord\bigg(\exp_{\psi}\bigg(\sum\limits_{a\in A_{+}}\frac{\mu(a)a(z_1)\dots a(z_n)t^{\deg_{\theta}(a)}}{a}\bigg)-1\bigg) \geq 1 .
\end{equation}
But by Proposition \ref{Prop5}, we know that the left hand side of \eqref{integrality} is in $A[z_1,\dots,z_n,t]$. Therefore, by the inequality in \eqref{integrality2}, we get the first part of the theorem. Evaluating \eqref{integrality} at $t=1$ gives the second part of the theorem.
\end{proof}
Recall that $\log_{\phi}=\sum_{j\geq 0} \gamma_{j}\tau^j$ is the logarithm series corresponding to $\phi$, and
$\log_{\psi}(g)=\sum_{j \geq 0}\gamma_j\ell_j(z_1)\dots \ell_j(z_n)t^j\tau^j(g)$ for any $g\in \widetilde{\TT_{n,t}}$ in the domain of $\log_{\psi}$. We define a positive integer $i(\phi)$ which satisfies 
\[\frac{\deg_{\theta}(\phi_{\theta,i(\phi)})-q^{i(\phi)}}{q^{i(\phi)}-1} \geq \frac{\deg_{\theta}(\phi_{\theta,s})-q^s}{q^s-1}\]
where $s$ runs over all the indices such that $\phi_{\theta,s} \neq 0$.
By \cite[Cor.~6.9]{EP14}, we have that 
\[\inorm{\gamma_j} \leq q^{\frac{q^j-1}{q^{i(\phi)}-1}(\deg_{\theta}(\phi_{\theta,i(\phi)})-q^{i(\phi)})}.\]
Since $\dnorm{\ell_{1}(z_k)}=q$, $\dnorm{\ell_{j}(z_k)}=q^{1+q+\dots+q^{j-1}}$ for all $j\geq 2$, and $1\leq k \leq n$, we have
\begin{equation}\label{order}
\dnorm{\gamma_j^{\prime}}=\dnorm{\gamma_j\ell_j(z_1)\dots \ell_j(z_n)t^j}\leq q^{\frac{q^j-1}{q^{i(\phi)}-1}(\deg_{\theta}(\phi_{i(\phi)})-q^{i(\phi)}+n(1+q+\dots + q^{i(\phi)-1}))}
\end{equation}
for all $j\geq 1$. Using \eqref{order} and properties of non-archimedian norm $\dnorm{\,\cdot\,}$, we deduce the following proposition.
\begin{proposition}\label{radiusofconv} The logarithm function $\log_{\psi}$ converges for all $f\in \widetilde{\TT_{n,t}}$ such that 
\[\dnorm{f} < q^{\frac{q^{i(\phi)}-\deg_{\theta}(\phi_{\theta,i(\phi)})-n(1+q+\dots + q^{i(\phi)-1})}{q^{i(\phi)}-1}}.\]
\end{proposition}
Now, we analyze the value of $L(\phi^{\vee},z_1,\dots,z_n,1)$ in the following result.
\begin{corollary}\label{analog}
If $0\leq n \leq q/r -(1+2\beta)$, then
\begin{equation}\label{integrality3}
\sum\limits_{a\in A_{+}}\frac{\mu(a)a(z_1)\dots a(z_n)t^{\deg_{\theta}(a)}}{a}=\log_{\psi}(1).
\end{equation}
In particular,
\begin{equation}\label{integrality4}
\sum\limits_{a\in A_{+}}\frac{\mu(a)a(z_1)\dots a(z_n)}{a}=\log_{\varphi}(1)=\frac{\log_{\phi}(\omega_n)}{\omega_n}.
\end{equation}
\end{corollary}
\begin{proof}
First we show that $1$ is within the radius of convergence of $\log_{\psi}$. We again note by \eqref{ineqcalc} that if $0\leq n \leq q/r -(1+2\beta)$, then $\beta+n \leq -1+(n+1)/q+(q-1)/r$ and $q > n$. Thus, by the choice of $\beta$ and the inequality $q>n$, we have that
\begin{equation}
\begin{split}
 \deg_{\theta}(\phi_{\theta,i(\phi)}) + n(1+q+\dots + q^{i(\phi)-1})&\leq \frac{q-1}{r} - (n+1)+\frac{n}{q}+\frac{1}{q}+n+n(q+\dots +q^{i(\phi)-1})\\
&< q + n(q+\dots +q^{i(\phi)-1}) \\
&\leq q+(q-1)(q+\dots +q^{i(\phi)-1})=q^{i(\phi)}.
\end{split}
\end{equation}
Therefore, by Proposition \ref{radiusofconv}, we see that $1$ is within the radius of convergence of $\log_{\psi}$. Finally, applying $\log_{\psi}$ to both sides of \eqref{integrality} finishes the first part. We now prove the last assertion. 
If we evaluate \eqref{integrality3} at $t=1$, then we get the first equality in \eqref{integrality4}. Now, recall the infinite product expansion of $\omega_{n}$ in \eqref{AndersonThakur}, and observe that $
\tau^i(\omega_n)=\ell_i(z_1)\ell_i(z_2)\cdots \ell_i(z_n)\omega_n$ for all $i\geq 1$. Thus,
\begin{equation}\label{partofproof}
\begin{split}
\log_{\phi}(\omega_n)&=\gamma_0\omega_n + \gamma_1\tau(\omega_n)+\gamma_2\tau^2(\omega_n)+\dots \\
&=\gamma_0\omega_n+\gamma_1\ell_1(z_1)\ell_1(z_2)\dots \ell_1(z_n)\omega_n+\gamma_2\ell_2(z_1)\ell_2(z_2)\dots \ell_2(z_n)\omega_n+\dots\\
&=\omega_n\log_{\varphi}(1),
\end{split}
\end{equation}
which gives the second equality in \eqref{integrality4}.
\end{proof}
\begin{remark}\label{calculation}
Observe that the second equality in Corollary \ref{Cor1} follows from the same calculation in \eqref{partofproof}  replacing $1$ by $P_{\phi}(z_1,\dots,z_n)$.
\end{remark}
\begin{remark} One can note that dealing with $\log_{\psi}$ which is not an entire function is one of the obstacles which bounds us from understanding \eqref{integrality4} for any $n$ and any Drinfeld $A$-module $\phi$ over $A$. It would be interesting to understand the value of $L(\phi^{\vee},z_1,\dots,z_n,1)$ and the polynomial $P_{\phi}(z_1,\dots,z_n)$ in general.
\end{remark}

\begin{example}\label{omegaotherdef} If we choose $\phi=\tilde{C}$ in Corollary \ref{analog} which is defined by $\tilde{C}_{\theta}:=\theta + \ell_1(z_1)\tau$, then we recover an immediate consequence of \cite[Prop. 5.9]{APTR} which can be stated as
\[
L(\tilde{C},\mathbb{A})=\sum_{a\in A_{+}} \frac{a(z_1)}{a} =\log_{\tilde{C}}(1)=\frac{\log_{C}(\omega_1)}{\omega_1}=-\frac{\tilde{\pi}}{(z_1-\theta)\omega_1},
\]
where the last equation follows from the fact that $\omega_1=\exp_{C}(\tilde{\pi}/(\theta-z_1))$ which was proved by Pellarin in \cite[\S~4]{Pellarin}.
\end{example}
\begin{remark}
Let $\psi$ be a Drinfeld $\tilde{\mathbb{A}}$-module as in Corollary \ref{analog} and recall that $\log_{\psi}=\sum_{i\geq 0}\gamma_i\ell_i(z_1)\dots \ell_i(z_n)t^{i}\tau^i$ is the logarithm series corresponding to $\psi$. If we compare the coefficients of $t^i$ for all $i$ on both sides of \eqref{integrality3}, we get 
\begin{equation}\label{lseries}
\sum\limits_{a\in A_{+,i}}\frac{\mu(a)a(z_1)\dots a(z_n)}{a}=\gamma_i\ell_i(z_1)\dots \ell_i(z_n).
\end{equation}
We note that \eqref{lseries} can be seen as the generalization of the formulas obtained by Perkins \cite[Thm. 4.16]{Perkins} which relates the coefficients of the logarithm function $\log_{C}$ to Pellarin $L$-Series.
\end{remark}
\begin{remark} We recall the definition of the Carlitz module $C$ from Section 2, and for any $x\in \CC_{\infty}$ in the domain of $\log_{C}$, let $\log_{C}(x)=\sum_{i\geq 0}L_{i}\tau^i(x)$. For integers $k,i\geq 0$, the power sum $S_i(k)$ is defined by  
\[
 S_i(k):=\sum_{a\in A_{+,i}}a^{k}.
\]
Then, Lee proved \cite[Thm. 4.1]{Lee} (see also \cite[Cor. 5.6.4(1)]{Thakur}) that 
\begin{equation}\label{powersum}
S_i(q^k-1)=\sum_{a\in A_{+,i}}a^{q^k-1}=
\begin{cases}
L_i(\theta^{q^k}-\theta)\dots (\theta^{q^k}-\theta^{q^{i-1}}) & \text{ if } k \geq i \\
0 & \text{ if } k < i.
\end{cases}
\end{equation}
We refer the reader to \cite{PapanikolasLog} and \cite[\S~5]{Thakur} for more details about the sum $S_i(k)$.

Now, let $\phi$ be a Drinfeld $A$-module of rank $r$ defined as in \eqref{intro0} such that $0\leq \beta \leq q/(2r)-1$, and consider the logarithm series $\log_{\phi}=\sum_{i\geq 0}\gamma_{i}\tau^i$ corresponding to $\phi$. If we define the Drinfeld $\tilde{\mathbb{A}}$-module $\psi$ by $\psi_{\theta}=\sum_{i=0}^{r}\phi_{\theta,i}\ell_{i}(z_1)t^i$, and evaluate $\log_{\psi}(1)=\sum_{i\geq 0}\gamma_{i}\ell_{i}(z_1)t^i$ at $z_1=\theta^{q^k}$ for some $k\in \NN$, then comparing coefficients of $t^i$ on both sides of \eqref{integrality3} for all $i\geq 0$ implies that  
\[
\sum_{a\in A_{+,i}}\mu(a)a^{q^k-1}=
\begin{cases}
\gamma_i(\theta^{q^k}-\theta)\dots (\theta^{q^k}-\theta^{q^{i-1}}) & \text{ if } k \geq i \\
0 & \text{ if } k < i
\end{cases}
\]
which can be seen as a generalization of \eqref{powersum}.
\end{remark}
\section{Entireness of the $L$-series $L(\phi^{\vee},z_1,\dots,z_n;x,y)$}
Before we state the main result of this section, we need a definition due to Goss.
\begin{definition} \cite[\S 8.5]{Goss} For each $y \in \mathbb{Z}_p$, let $g_y(z)$ be a power series in $z$ such that $g_y$ converges at all values of $z\in \CC_{\infty}$. Let $(x,y)\in \mathbb{S}_{\infty}$. We call the entire power series $f(x,y):=g_y(1/x)$ \textit{an entire function on $\mathbb{S}_{\infty}$} if for any bounded subset $H \subset \CC_{\infty}$ and $\epsilon >0$, there exists $\delta_{H}>0$ such that if $y_0,y_1 \in \mathbb{Z}_p$ and $|y_0-y_1|_{p}<\delta_{H}$, then $\inorm{g_{y_0}(z)-g_{y_1}(z)} <\epsilon$ for all $z\in H$.
\end{definition}
This section occupies the proof of the following result.
\begin{theorem}\label{entire} The infinite series $L(\phi^{\vee},z_1,\dots,z_n;x,y)$ can be analytically continued to an entire function on $\mathbb{C}^n_{\infty} \times \mathbb{S}_{\infty}$. In particular, the $L$-series $L(\phi^{\vee},z_1,\dots,z_n,s)$ converges in $\TT_n$ for all $s\in \mathbb{Z}$.
\end{theorem}
In order to prove Theorem \ref{entire}, we first introduce the following setting in \cite[\S~8]{APTR} which was given by Angl\`{e}s, Pellarin, and Tavares Ribeiro. We also use Proposition \ref{Prop5} to prove Theorem \ref{logalgthm}. 

Let $Y_1,\dots,Y_n$ be indeterminates. We set 
\[
 \mathbb{B}_n:=\CC_{\infty}[Y_1,\dots,Y_n,\tau(Y_1),\dots,\tau(Y_n),\dots],
\]
where the action of $\tau$ on $\mathbb{B}_n$ is given by $\tau \cdot \tau^i(Y_j)=\tau^{i+1}(Y_j)$. Let $W$ be another indeterminate. We define $\mathbb{D}_n$ to be the ring of elements of the form $\sum_{i\geq 0}b_i\tau^i(W)$, where $b_i\in \mathbb{B}_n.$ For elements $F=\sum f_i\tau^i(W) $ and $G=\sum g_i\tau^i(W)$ in $\mathbb{D}_n$, the multiplication $F\cdot G$ is given by $\sum_{k \geq 0}\sum_{i+j=k}f_i\tau^i(g_j)\tau^k(W)$, and the action of $\tau$ on $\tau^i(W)$ is equal to $\tau^{i+1}(W)$. Using the Drinfeld $A$-module $\phi$ in \eqref{intro0}, define
\[ 
 \mathcal{E}_{\phi}(Y_1,\dots,Y_n,W):=\sum_{a\in A_{+}} \frac{\mu(a)C_a(Y_1)\dots C_a(Y_n)}{a}\tau^{\deg_{\theta}(a)}(W).
\]
The multiplication rule of the elements in $\mathbb{D}_{n}$ by any element $\gamma \in K$, $z_j$ for all $1\leq j\leq n$,  and $t$ is given as follows:
\begin{equation}\label{multiplication}
\begin{split}
\gamma \cdot f&=\gamma f {\text{ for any $f\in \mathbb{D}_n$}}\\
z_j \cdot \tau^m(Y_i)&=\tau^m(Y_i) {\text{ if $i \neq j$}} \\
z_j \cdot \tau^m(Y_j)&=\tau^m(C_{\theta}(Y_j)) {\text{ for all $1\leq j \leq n$}} \\
t \cdot \tau^m(Y_j)&=\tau^m(Y_j) {\text{  for all $1\leq j \leq n$}} \\
z_j \cdot \tau^m(W)&=\tau^m(W) {\text{ for all $1\leq j \leq n$}}\\
t \cdot \tau^m(W)&=\tau^{m+1}(W).
\end{split}
\end{equation}
Moreover, for any $f\in \mathbb{D}_{n}$ and $g\in K[z_1,\dots,z_n,t]$ given by
\[
g=\sum_{j_1,\dots,j_n,j_{n+1}\geq 0}g_{j_1\dots j_{n+1}}z_1^{j_1}\dots z_n^{j_{n}}t^{j_{n+1}},
\]
we define the $K[z_1,\dots,z_n,t]$-action on $\mathbb{D}_{n}$ by
\[
g\cdot f=\sum_{j_1,\dots,j_n,j_{n+1}\geq 0}(z_1^{j_1}\cdot f)\dots (z_n^{j_n}\cdot f)(t^{j_{n+1}}\cdot f).
\]
Thus, the ring $\mathbb{D}_n$ is now a $K[z_1,\dots,z_n,t]$-algebra. Using \eqref{multiplication}, we conclude that for any $a=a(\theta)\in A$ and any $m\geq 1$, we have
\begin{equation}\label{zactions}
(a(z_1)\dots a(z_n))\cdot \tau^m(Y_1\dots Y_n)=\tau^m(C_a(Y_1)\dots C_a(Y_n)).
\end{equation}
The action of $K[z_1,\dots,z_n]$ to $\mathbb{D}_n$ extends to an action of $K[z_1,\dots,z_n][[t]]$ in a following way: If $G=\sum G_it^i \in K[z_1,\dots,z_n][[t]]$, we set
\[G\cdot (Y_1\dots Y_nW):=\sum G_i\cdot (Y_1\dots Y_n\tau^i(W)).\]
We have the following useful lemma for the action defined above.
\begin{lemma}[{Angl\`{e}s, Pellarin, Tavares Ribeiro \cite[Lem.~3.11(2)]{AnglesPellarinTavaresRibeiro}}]\label{action}
The map $g:K[z_1,\dots,z_n,t] \rightarrow \mathbb{D}_n$ defined by 
\[ g(f)=f \cdot (Y_1\dots Y_n) \]
is injective for all $f\in K[z_1,\dots,z_n,t]$.
\end{lemma}
By Theorem \ref{canonicaldeformationTaelman}, the Taelman $L$-value $L(\psi,\tilde{\mathbb{A}})$ corresponding to Drinfeld $\tilde{\mathbb{A}}$-module $\psi$ defined in \eqref{psi} is given by 
\[
L(\psi,\tilde{\mathbb{A}})=\sum_{a\in A_{+}}\frac{\mu(a)a(z_1)\dots a(z_n)t^{\deg_{\theta}(a)}}{a},
\]
and by \eqref{multiplication} and \eqref{zactions}, we get $L(\psi,\tilde{\mathbb{A}})\cdot (Y_1\dots Y_nW)=\mathcal{E}_{\phi}(Y_1,\dots,Y_n,W)$.
By the action defined in \eqref{multiplication}, we notice that for any $1\leq j \leq n$,
\begin{equation}\label{zjaction}
\begin{split}
t^i\ell_{i}(z_j)\cdot Y_jW&=t^i((z_j-\theta^{q^{i-1}})\dots(z_j-\theta^q) (z_j-\theta))\cdot Y_jW\\
&=t^i(z_j-\theta^{q^{i-1}})\dots (z_j-\theta^{q^2})(z_j-\theta^q)\cdot(C_{\theta}(Y_j)-\theta Y_j)W \\
&=t^i(z_j-\theta^{q^{i-1}})\dots (z_j-\theta^{q^2})(z_j-\theta^q)\cdot \tau(Y_j)W \\
&=t^i(z_j-\theta^{q^{i-1}})\dots (z_j-\theta^{q^2})\cdot (\tau(C_{\theta}(Y_j))-\theta^qY_j)W\\
&=t^i(z_j-\theta^{q^{i-1}})\dots (z_j-\theta^{q^2})\cdot\tau^2(Y_j)W\\
&\vdots \\
&=t^i\tau^i(Y_j)W\\
&=\tau^i(Y_jW),
\end{split}
\end{equation}
and therefore \eqref{zjaction} implies that
\begin{equation}\label{zjaction2}
\begin{split}
t^i\ell_{i}(z_j)\cdot (Y_1Y_2\dots Y_{j-1} Y_j Y_{j+1} \dots Y_nW)&=t^i\ell_{i}(z_j)\cdot (Y_jY_1\dots Y_{j-1} Y_{j+1} \dots Y_nW)\\
&=t^i\cdot\tau^i(Y_j)Y_1\dots Y_{j-1} Y_{j+1} \dots Y_nW\\
&=\tau^i(Y_jW)Y_1\dots Y_{j-1} Y_{j+1} \dots Y_n.
\end{split}
\end{equation}
Combining \eqref{zjaction} and \eqref{zjaction2} gives that
\begin{equation}\label{zjaction3}
t^i\ell_{i}(z_1)\dots \ell_{i}(z_n)\cdot (Y_1Y_2\dots Y_{j-1}Y_j Y_{j+1} \dots Y_n) =\tau^i(Y_1\dots Y_nW).
\end{equation}
Since $\exp_{\psi}=\sum_{i} \alpha_i t^i\ell_i(z_1)\dots \ell_i(z_n) \tau^i$, \eqref{zactions} and \eqref{zjaction3} imply that 
\[
\exp_{\phi}(\mathcal{E}_{\phi}(Y_1,\dots,Y_n,W))=\exp_{\psi}(L(\psi,\tilde{\mathbb{A}}))\cdot (Y_1\dots Y_nW). 
\]
But by Proposition \ref{Prop5}, $\exp_{\psi}(L(\psi,\tilde{\mathbb{A}}))\in A[z_1,\dots,z_n,t]$. Therefore, we have that 
\begin{equation}\label{multivariable}
\exp_{\phi}(\mathcal{E}_{\phi}(Y_1,\dots,Y_n,W)) \in A[Y_1,\dots,Y_n,W]. 
\end{equation}
Now for any $m \in \NN$ and indeterminates $X_1,X_2,\dots, X_n,w$, choose a suitable $\CC_{\infty}$-algebra homomorphism sending $\tau^m(Y_j)$ to $X_j^{q^m}$ for all $1\leq j\leq n$ and $\tau^m(W)$ to $w^{q^m}$. Rewriting \eqref{multivariable} with this homomorphism, we get the following theorem.
\begin{theorem}\label{logalgthm}Let $\phi$ be a Drinfeld $A$-module as in \eqref{intro0}, and let $X_1,\dots, X_n,w$ be indeterminates. The series
\[\exp_{\phi}\bigg(\sum_{a\in A_{+}}\frac{\mu(a)C_a(X_1)\dots C_a(X_n)}{a}w^{q^{\deg_{\theta}(a)}}\bigg) \in K[X_1,\dots,X_n][[w]]
\]
is actually in $A[X_1,\dots,X_n,w]$.
\end{theorem}
We define 
\[
S^{\phi}_k=\sum_{a\in A_{+,k}}\frac{\mu(a)C_{a}(X_1)\dots C_{a}(X_n)}{a}
\]
and
\[\jmath^{\phi}=\sum\limits_{a\in A_{+}}\frac{\mu(a)C_{a}(X_1)\dots C_a(X_n)}{a}w^{q^{\deg(a)}}=\sum_{k\geq 0}S^{\phi}_kw^{q^k}.
\]
For $\exp_{\phi}=\sum_{j\geq 0}\alpha_j\tau^j$, we set $Z^{\phi}_k=\sum\limits_{i=0}^k (S^{\phi}_{k-i})^{q^i}\alpha_i$,
and therefore
\begin{equation}\label{seriesZk}
\exp_{\phi}(\jmath^{\phi})=\sum_{k \geq 0}Z^{\phi}_kw^{q^k}.
\end{equation}
By Theorem \ref{logalgthm} and \eqref{seriesZk}, we have $Z^{\phi}_k \in A[X_1,\dots,X_n]$ for all $k\geq 0$, and for $k$ arbitrarily large, $Z^{\phi}_k=0$. 

We recall the definition of the Carlitz module $C$ from Section 2 and that $\Ker(\exp_{C})=\tilde{\pi}A$ where $\tilde{\pi} \in \CC_{\infty}^{\times}$. Let $\tilde{X}$ be the set of indeterminates $X_1,\dots,X_n$. For indeterminates $X_1,\dots X_n$, we denote the polynomial ring $\CC_{\infty}[X_1,\dots,X_n]$ by $\CC_{\infty}[\tilde{X}]$. Following the work of Angl\`{e}s, Pellarin, and Tavares Ribeiro in \cite[\S~3]{AnglesPellarinTavaresRibeiro}, for any $G\in \CC_{\infty}[\tilde{X}]$, we define a function $\dnorm{\,\cdot\,}_{\tilde{X}}\colon \CC_{\infty}[\tilde{X}] \to \RR $ by 
\[
\dnorm{G}_{\tilde{X}} =\sup\{ \inorm{G(x_1,\dots,x_n)} \mid x_i\in \exp_C(K_{\infty}\tilde{\pi})\quad \text{for } 1\leq i \leq n\},
\]
so that $\dnorm{\,\cdot\,}_{\tilde{X}}$ defines a well-defined and ultrametric norm on $\CC_{\infty}[\tilde{X}]$ (see \cite[\S~3]{AnglesPellarinTavaresRibeiro}). 
\begin{proposition}\label{vanishing} Let $\phi$ be a Drinfeld $A$-module of rank $r$ as in \eqref{intro0} and $\beta=\max\{\deg_{\theta}(\phi_{\theta,i}) \mid 1\leq i \leq r\}$. Then $Z^{\phi}_k=0$ when $k > r(n + \beta)/(q-1).$
\end{proposition}
\begin{proof}
We apply the similar methods in the proof of \cite[Thm.~6.2]{ChangEl-GuindyPapanikolas}. We know by \cite[Lem.~3.1(1)]{AnglesPellarinTavaresRibeiro} that $\dnorm{X_1\dots X_n}_{\tilde{X}}=q^{n/q-1}.$ Note that by \cite[\S 3.1]{AnglesPellarinTavaresRibeiro}, for all $a\in A$ and $G\in \CC_{\infty}[\tilde{X}]$, the action $a \cdot G =G(C_{a}(X_1),\dots,C_{a}(X_n))$ is isometric with respect to the norm $\dnorm{\,\cdot\,}_{\tilde{X}}$. Thus, we have 
\begin{equation}\label{boundnorm}
\dnorm{C_a(X_1)C_a(X_2)\dots C_a(X_n)}_{\tilde{X}}=\dnorm{X_1\dots X_n}_{\tilde{X}}=q^{n/q-1}.
\end{equation}
By \cite[Eq. 28]{ElGuindyPapanikolas}, we get
\begin{equation}\label{boundexp}
\inorm{\alpha_k}\leq q^{q^k(\beta/(q-1)-k/r)}.
\end{equation}
The definition of $S^{\phi}_k$, Lemma \ref{mu}(b) and \eqref{boundnorm} imply
\begin{equation}\label{bounds}
\dnorm{S^{\phi}_k}_{\tilde{X}}\leq \max_{a\in A_{+,k}} \bigg \{ \biggr\lVert\frac{\mu(a)}{a}\biggr\rVert\dnorm{X_1\dots X_n}_{\tilde{X}}\bigg\} \leq q^{-k/r}q^{n/q-1}.
\end{equation}
Therefore, \eqref{boundexp} and \eqref{bounds} imply that
\[
\dnorm{Z^{\phi}_k}_{\tilde{X}} \leq \max_{0 \leq i \leq k} \{ \dnorm{S^{\phi}_{k-i}}_{\tilde{X}}^{q^{i}} \inorm{\alpha_i}\} \leq (q^{n/q-1})^{q^k}q^{q^k(\beta/(q-1)-k/r)}=q^{q^k(\frac{n+\beta}{q-1}-\frac{k}{r})}.
\]
Finally, when $k > r(n + \beta)/(q-1)$, we have $ \dnorm{Z^{\phi}_k}_{\tilde{X}} < 1$ which implies that $Z^{\phi}_k=0$ by \cite[Lem.~4.3]{AnglesPellarinTavaresRibeiro}.
\end{proof}

\begin{lemma}\label{partialanalog} For $ k > r(n + \beta)/(q-1)$, the series
$
H_{k,n-1}:=\sum_{a\in A_{+,k}} \mu(a)a(z_1)\dots a(z_{n-1}) 
$
vanishes. In particular, $L(\phi^{\vee},z_1,\dots,z_{n-1},0) \in A[z_1,\dots,z_{n-1}].$
\end{lemma}
\begin{proof}
The proof is using the ideas of the proof of \cite[Lem.~4.13]{AnglesPellarinTavaresRibeiro}. If we see $S^{\phi}_k$ as a polynomial of $X_n$ without any constant term, and note that $C_a(X_n)\equiv aX_n \pmod{X^q_n}$, we have
\begin{equation}
\begin{split}
Z_k^{\phi} \equiv S_{k}^{\phi}&\equiv \sum\limits_{a\in A_{+,k}} \frac{\mu(a)C_a(X_1)\dots C_a(X_n)}{a} \pmod{X_n^q} \\
&\equiv \sum\limits_{a\in A_{+,k}}\frac{\mu(a)C_a(X_1)\dots C_a(X_{n-1})aX_n}{a} \pmod{X_n^q}\\
&\equiv X_n \sum_{a\in A_{+,k}}\mu(a)C_a(X_1)\dots C_a(X_{n-1}) \pmod{X^q_n}. 
\end{split}
\end{equation}
For $ k > r(n + \beta)/(q-1)$, we know by Proposition \ref{vanishing} that $Z_k^{\phi}=0$ which gives us that 
\begin{equation}\label{vanishingintheproof}
\sum_{a\in A_{+,k}}\mu(a)C_a(X_1)\dots C_a(X_{n-1})=0.
\end{equation}
Moreover by \eqref{zactions}, we have that
\begin{equation}\label{vanishingintheproof2}
 \sum_{a\in A_{+,k}}\mu(a)C_a(X_1)\dots C_a(X_{n-1})= \sum_{a\in A_{+,k}}\mu(a)a(z_1)\dots a(z_{n-1})\cdot(X_1\dots X_{n-1}).
\end{equation}
By Lemma \ref{action}, \eqref{vanishingintheproof} and \eqref{vanishingintheproof2}, we conclude that 
\[\sum_{a\in A_{+,k}}\mu(a)a(z_1)\dots a(z_{n-1})=0\]
for $ k >r(n + \beta)/(q-1)$. On the other hand, we have
\[L(\phi^{\vee},z_1,\dots,z_{n-1},0)=\sum\limits_{k\geq 0}\sum_{a\in A_{+,k}}\mu(a)a(z_1)\dots a(z_{n-1}).\]
Therefore, the second part of the proposition follows from the first part.
\end{proof}
\begin{remark}
We note that Lemma \ref{partialanalog} can be seen as a corollary of Simon's Lemma \cite[Lem. 4]{AnglesPellarin} when $\phi=C$.
\end{remark}
\begin{lemma}\label{bound2}For $d \geq 3r + r(n+1+\beta)/(q-1)$ and $n >0$, we have
\[ 
 \dnorm{\mathcal{L}_{d,n}(x,y)} \leq \inorm{x}^{-d}q^{d(1-\frac{1}{r})}q^{-q^{\big[\frac{d}{r} - \frac{n+1+\beta}{q-1}\big]-2}},
\]
where $[\,\cdot\,]$ is the integer part function.
\end{lemma}
\begin{proof}
We adapt the ideas in \cite[Lem.~7]{AnglesPellarin}. Recall from Section 1 that $p=q^l$. For any $w\in \mathbb{Z}_p$ such that $w=\sum w_i q^i$ where $0\leq w_i\leq q-1$, set $\ell_q(w):=\sum w_i$. Let $y\in \ZZ_{p}$, and $\sum_{j\geq 0}b_jp^j$ be the $p$-adic expansion of $y$ such that $0\leq b_j \leq p-1$ for all $j$, and for $m \geq 1$, set $y_{m}:=\sum_{i=0}^{lm-1}b_kp^k$ such that $\sum_{j=0}^{m}v_jq^j$ be the $q$-adic expansion of $y_m$ where $0\leq v_j \leq q-1$. We have that
\[
\mathcal{L}_{d,n}(x,y_m)=\frac{1}{x^d\theta^{dy_m}}H_{d,n+\ell_q(y_m)}(z_1,\dots,z_n,\theta,\dots,\theta,\dots,\theta^{q^m},\dots,\theta^{q^m}),
\]
where $\theta^{q^i}$ appears $v_i$ many times. By the bound on coefficients $v_j$, it is easy to see that 
\[
\ell_q(y_m)=\sum_{j=0}^mv_j\leq (m+1)(q-1).
\]
By Lemma \ref{partialanalog}, if $d > r(n+1+(m+1)(q-1)+\beta)/(q-1)=r(m+1)+r(n+\beta+1)/(q-1)$, then $H_{d,n+\ell_q(y_m)}(x,y_m)=0$. Notice that by the definition of $y_m$ and \eqref{Lucasformula}, we have that
\[
\binom{y}{j}=\binom{y_m}{j} 
\]
for $j=0,1,\dots, q^{m}-1$. Therefore, for any $a\in A_{+,d}$,
\begin{equation}\label{binomial}
\bigg|\sum_{j \geq 0}\bigg( \binom{y}{j} - \binom{y_m}{j} \bigg)(\langle a \rangle -1)^{j}\bigg|_{\infty} \leq q^{-q^{m}}. 
\end{equation}
Thus by Lemma \ref{mu}(b) and \eqref{binomial}, we have that
\begin{equation}\label{boundonL}
\begin{split}
\dnorm{\mathcal{L}_{d,n}(x,y)-\mathcal{L}_{d,n}(x,y_m)} &= \biggl\lVert x^{-d}\sum\limits_{a\in A_{+,d}} \mu(a)a(z_1)\dots a(z_n)\sum\limits_{i \geq 0}\bigg( \binom{y}{i} - \binom{y_m}{i} \bigg)(\langle a \rangle -1)^{i}\biggr\rVert  \\
&\leq \inorm{x}^{-d}q^{d(1-\frac{1}{r})}q^{-q^{m}}.\\
\end{split}
\end{equation}
If we choose $m+2=[d/r - (n+\beta+1)/(q-1)]\geq 3$, then 
\[
\mathcal{L}_{d,n}(x,y_m)=\frac{1}{\theta^{dy_m}}H_{d,n+\ell_q(y_m)}(z_1,\dots,z_n,\theta,\dots,\theta,\dots,\theta^{q^m},\dots,\theta^{q^m})=0.
\]
Therefore, the result follows from \eqref{boundonL} and our choice for $m$ above.
\end{proof}
\begin{proof}[{Proof of Theorem \ref{entire}}]  Let $(x,y)\in \mathbb{S}_{\infty}$. For any $\xi_i \in \mathbb{C}_{\infty}$ and arbitrarily large $d$, we have by Lemma \ref{bound2} that
\begin{equation}\label{norm}
\biggl\lVert\sum\limits_{a\in A_{+,d}} \mu(a)a(\xi_1)\dots a(\xi_n)\langle a \rangle^y \biggr\rVert \leq \inorm{\xi_1\dots\xi_n}^dq^{d(1-\frac{1}{r})}q^{-q^{\big[\frac{d}{r} - \frac{n+1+\beta}{q-1}\big]-2}}.
\end{equation}
Thus, we see that as $d$ goes to $\infty$, the right hand side of (\ref{norm}) approaches to 0. On the other hand, we have that 
\[L(\phi^{\vee},z_1,\dots,z_n;x,y)=\sum_{d \geq 0} x^{-d}\sum\limits_{a\in A_{+,d}} \mu(a)a(z_1)\dots a(z_n)\langle a \rangle^y.\]
Therefore, by \cite[Thm.~8.5.7]{Goss}, the series $L(\phi^{\vee},z_1,\dots,z_n;x,y)$ is entire on $\CC_{\infty}^{n} \times \mathbb{S}_{\infty}$ which proves the theorem.
\end{proof}

\begin{remark} A natural direction at this point would be understanding the zeros of the $L$-series $L(\phi^{\vee},z_1,\dots,z_n;x,y)$. When $\phi$ is the Carlitz module $C$, we refer the reader to results explained in \cite{Goss}. For a Drinfeld $A$-module of higher rank, studying zeroes of $L(\phi^{\vee},z_1,\dots,z_n;x,y)$ seems challenging because of the function $\mu\colon A_{+} \to A$ whose nature has not been discovered deeply yet. We hope to turn back this problem in future.
\end{remark}

\end{document}